\documentclass[12pt,a4paper]{amsart}

\usepackage{amsfonts}

\usepackage{amsthm}

\usepackage{amsmath}

\usepackage{rotating}
\usepackage{tikz}

\usepackage{amscd}

\usepackage[latin2]{inputenc}

\usepackage{t1enc}

\usepackage[mathscr]{eucal}

\usepackage{indentfirst}

\usepackage{graphicx}

\usepackage{graphics}

\usepackage{pict2e}

\usepackage{mathrsfs}

\usepackage{hyperref} %this gives clickable references, which is nice
\numberwithin{equation}{section}

\def\red{\textcolor{red}}

\theoremstyle{plain}

\newtheorem{Th}{Theorem}[section]

\newtheorem{Lemma}[Th]{Lemma}

\newtheorem{Cor}[Th]{Corollary}

\newtheorem{Pro}[Th]{Proposition}

\theoremstyle{definition}

\newtheorem{Def}[Th]{Definition}

\newtheorem{Exa}[Th]{Example}

\newtheorem{Conj}[Th]{Conjecture}

\newtheorem{Rem}[Th]{Remark}

\newtheorem{?}[Th]{Problem}

\newcommand{\Hom}{{\rm{Hom}}}
\newcommand{\End}{{\rm{End}}}

\newcommand{\diam}{{\rm{diam}}}

\begin{document}

\title{Graph homomorphisms between trees}

\author[P. Csikv\'ari]{P\'{e}ter Csikv\'{a}ri}

\address[P\'{e}ter Csikv\'{a}ri]{E\"{o}tv\"{o}s Lor\'{a}nd University
  \\ Department of Computer  Science \\ H-1117 Budapest
\\ P\'{a}zm\'{a}ny P\'{e}ter s\'{e}t\'{a}ny 1/C \\ Hungary \& Alfr\'ed 
R\'enyi Institute of Mathematics \\ H-1053 Budapest \\ Re\'altanoda
u. 13-15. \\ Hungary} 

\email{csiki@cs.elte.hu}

\author[Z. Lin]{Zhicong Lin}
\address[Zhicong Lin]{
Department of Mathematics and Statistics, Lanzhou University, China \&
Institut Camille Jordan, UMR 5208 du CNRS, Universit\'{e} de Lyon,
Universit\'{e} Lyon 1, France} 
\email{lin@math.univ-lyon1.fr}

\thanks{The first author  is partially supported by the
Hungarian National Foundation for Scientific Research (OTKA), grant
no. K81310 and by grant no.\ CNK 77780 from the National Development
Agency of Hungary, based on a source from the Research and Technology
Innovation Fund. He is also partially supported by MTA R\'enyi
"Lend\"ulet" Groups and Graphs Research Group.}
\thanks{
The second author is supported by the China Scholarship Council (CSC) for
studying abroad, CSC no. 2010618092.}

\begin{abstract}
In this paper we study several problems concerning the number of
homomorphisms of trees. We give an algorithm for the number of homomorphisms
from a tree to any graph by the Transfer-matrix method. By using this
algorithm and some transformations on trees, we study various extremal
problems about the number of homomorphisms of trees.  These applications
include a far reaching generalization of Bollob\'as and Tyomkyn's result
concerning the number of walks in trees. 
 
Some other highlights of the paper are the following. Denote by $\hom(H,G)$
the number of homomorphisms from a graph $H$ to a graph $G$.  

For any tree $T_m$ on $m$ vertices we give a general lower bound for
$\hom(T_m,G)$ by certain entropies of Markov chains defined on the graph
$G$. As a particular case, we show that for any graph $G$, 
$$\exp(H_{\lambda}(G))\lambda^{m-1}\leq\hom(T_m,G),$$ 
where $\lambda$ is the largest eigenvalue of the adjacency matrix of $G$ and
$H_{\lambda}(G)$ is a certain constant depending only on $G$ which we call the
spectral entropy of $G$. In the particular case when $G$ is the path $P_n$ on
$n$ vertices, we prove that 
$$\hom(P_m,P_n)\leq \hom(T_m,P_n)\leq \hom(S_m,P_n),$$
where $T_m$ is any tree on $m$ vertices, and $P_m$ and $S_m$ denote the path
and star on $m$ vertices, respectively.

We also show that if $T_m$ is any fixed tree and
$$\hom(T_m,P_n)>\hom(T_m,T_n),$$
for some tree $T_n$ on $n$ vertices, then $T_n$ must be the tree obtained from a path $P_{n-1}$ by attaching a pendant vertex to the second vertex of
$P_{n-1}$. In fact, we conjecture that if 
$n\geq 5$ and $T_n$ is an arbitrary tree on $n$ vertices, then
$$\hom(T_m,P_n)\leq \hom(T_m,T_n)$$
for any tree $T_m$.

All the results together enable us to show that 
$$
|\End(P_m)|\leq|\End(T_m)|\leq|\End(S_m)|,
$$
where $\End(T_m)$ is the set of all endomorphisms of $T_m$ (homomorphisms from
$T_m$ to itself). 
\end{abstract}

\maketitle

\tableofcontents

\section{Introduction}

We use standard notations and terminology of graph theory, see for instance
\cite{bb,bm}.  The graphs considered here are finite and undirected 
without multiple edges and loops. Given a graph $G$, we write $V(G)$ for the
vertex set and $E(G)$ for the edge set. A \emph{homomorphism} from a graph $H$
to a graph $G$ is a mapping $f:V(H)\to V(G)$  
such that the images of adjacent vertices are adjacent. Denote by $\Hom(H,G)$
the set of homomorphisms from $H$ to $G$ and by $\hom(H,G)$ the number of
homomorphisms from $H$ to $G$.  Throughout this article, we write $P_{n}$ and
$S_{n}$ for the path and the star on $n$ vertices, respectively. The length of
a path is the number of its edges. The {\em union of  graphs} $G$ and $H$ is
the graph $G\cup H$ with vertex set  $V (G)\cup V (H)$ and edge set $E(G)\cup
E(H)$. A tree $T$ together with a root vertex $v$ will be denoted by $T(v)$.

The problem of computing $\hom(H,G)$ is difficult in general. However, there
has been recent interest in counting homomorphisms between special
graphs. In particular, formulas for computing the number of homomorphisms
between two different paths were given in \cite{aw, lz}. But even for these
special trees, the formulas are bulky and inelegant. In
Section~\ref{treewalk}, by using the Transfer-matrix method, we shall give an
algorithm for computing the number of homomorphisms from trees to any
graph. This algorithm will be called {\em Tree-walk algorithm}.

Recently, the first author proved a conjecture of Nikiforov concerning the
number of closed walks on trees. He proved in \cite{pc} that, for a fixed
integer $m$, the number of closed walks of length $m$ on trees of order
$n$ attains its maximum at the star $S_n$ and its minimum at the path $P_n$. 
In other words,
\begin{equation}\label{cycle}
\hom(C_{m},P_n)\leq \hom(C_{m},T_n)\leq \hom(C_{m},S_n),
\end{equation}
where $T_n$ is a tree on $n$ vertices and $C_{m}$ is the cycle on $m$
vertices.  

Bollob\'as and Tyomkyn \cite{bt} gave a variant of the first author's result
by replacing the number of  closed walks by  the number of all walks, that is 
\begin{equation}\label{path}
\hom(P_{m},P_n)\leq \hom(P_{m},T_n)\leq \hom(P_{m},S_n),
\end{equation}
where $T_n$ is a tree on $n$ vertices.
In both \cite{bt} and \cite{pc}, the authors use a certain transformation of
trees. In \cite{pc}, it is called the \emph{generalized tree shift}, whereas
in \cite{bt}, it is renamed to \emph{KC-transformation}.

\begin{figure}[h!] 
\begin{center}
\scalebox{.65}{\includegraphics{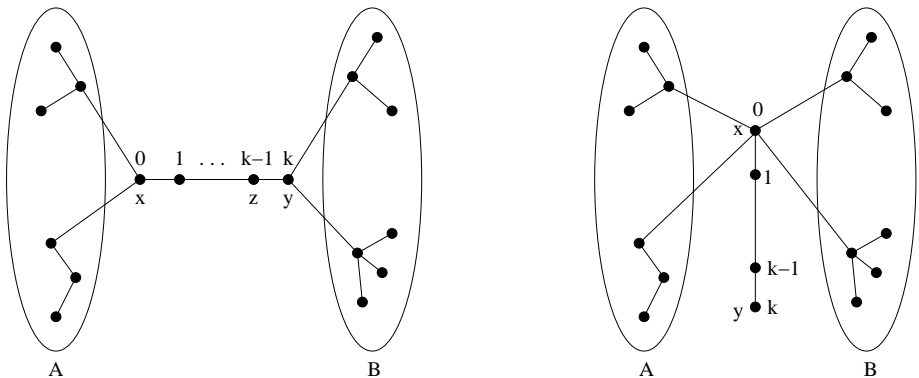}}  
\end{center}
\caption{\label{KC} The KC-transformation.} 
\end{figure}

To define this transformation, let $x$ and $y$ be two vertices of a tree $T$
such that every interior vertex of the unique $x$--$y$ path $P$ in $T$ has
degree two, and write $z$ for the neighbor of $y$ on this path. Denote by
$N(v)$ the set of neighbors of a vertex $v$. The KC-transformation,
$KC(T,x,y)$, of the tree $T$ with respect to the path $P$ is obtained from $T$
by deleting all edges between $y$ and $N(y)\setminus z$ and adding the edges
between $x$ and $N(y)\setminus z$ instead (See Fig.~\ref{KC}). Note that
$KC(T,x,y)$ and $KC(T,y,x)$ are isomorphic.

The following property of
KC-transformation was proved in~\cite{pc}. 

\begin{Pro}\label{pr1}
The KC-transformation gives rise to a graded poset of trees on $n$ vertices
with the star as the largest and the path as the smallest element. See
Figure~\ref{poset}. 
\end{Pro}

\begin{figure}[h!] 
\begin{center}
\scalebox{.65}{\includegraphics{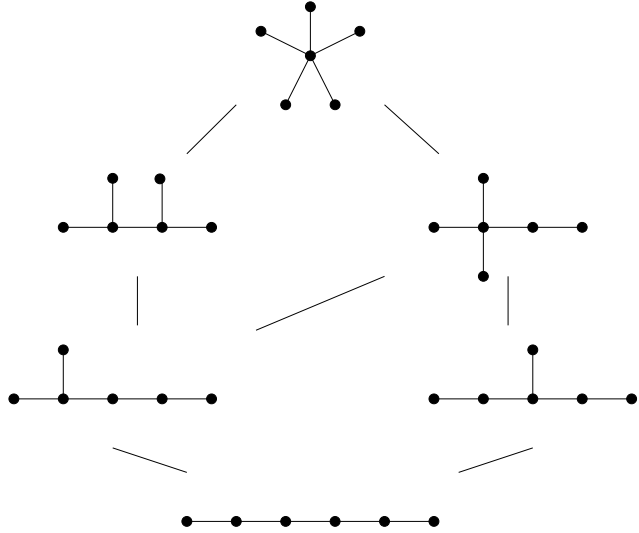}}\caption{The induced poset of
  KC-transformation on trees of $6$ vertices. \label{poset}}     
\end{center}
\end{figure}

In \cite{pc} the first author proved that the KC-transformation increases the
number of closed walks of fixed length in trees. By Proposition~\ref{pr1},
this leads to the proof of inequality ~\eqref{cycle}.

In the very same spirit,  Bollob\'as and Tyomkyn~\cite{bt} showed that the
KC-transformation increases the number of walks of fixed length in trees. In
the language of graph homomorphism, their result can be
restated as follows.  

\begin{Th} [Bollob\'as-Tyomkyn] \label{th:1}
Let $T$ be a tree and let $T'$ be obtained from $T$  by a
KC-transformation. Then  
\begin{equation} \label{eq:1}
\hom(P_m,T')\geq\hom(P_m,T)
\end{equation} 
for any $m\geq 1$.
\end{Th}

Now a natural question arises: does inequality~\eqref{eq:1} still true when
replacing $P_m$ by any tree? A tree is called \emph{starlike} if it has
at most one vertex of degree greater than two. Note that paths are
starlike. We answer this question in the affirmative for starlike trees.

\begin{Th}\label{of starlike}
Let $T$ be a tree and $T'$ the KC-transformation of $T$ with respect to a path
of length $k$. Then the inequality  
\begin{equation}\label{eq:2}
\hom(H,T')\geq \hom(H,T)
\end{equation}
holds when  $k$ is even and $H$ is any tree, or $k$ is odd and $H$ is a starlike
tree. 
\end{Th} 
 
Moreover, we find a counterexample for inequality~\eqref{eq:2} when $k$ is odd
and $H$ is not a starlike tree (see the end of Section~\ref{treewalks}).  
Another extremal problem concerning the number of homomorphisms between trees
that worth considering is to find the extremal trees for $\hom(\cdotp, P_n)$
over all trees on $m$ vertices. For this problem, we prove the following
theorem, which can be considered as a dual of inequality~\eqref{path}.

\begin{Th} \label{into paths} Let $T_m$ be  a tree on $m$ vertices and let
  $T'_m$ be obtained from $T_m$ by a KC-transformation.
\begin{itemize}
\item[(i)] If $n$ is even, or $n$ is odd and $\diam(T_m)\leq n-1$, then
\begin{equation}\label{KCtopath}
\hom(T_m,P_n)\leq \hom(T'_m,P_n).
\end{equation}
\item[(ii)] For any $m,n$,
$$\hom(P_m,P_n)\leq \hom(T_m,P_n)\leq \hom(S_m,P_n).$$
\end{itemize}
\end{Th}

Note that inequality~\eqref{KCtopath} is not  true in general when $n$ is odd
and $\diam(T_m)$ is greater than $n-1$; see Fig.~\ref{counter2} for a
counterexample.

For the sake of keeping this paper self-contained, we will also give a new
proof for the following theorem of Sidorenko \cite{si} concerning the extremal
property of the stars among trees.  Note that Fiol and Garriga~\cite{fg}
proved the special case of this theorem when $T_m=P_m$, clearly, they were not
aware of the work of Sidorenko. 
  
\begin{Th}[Sidorenko] \label{of stars} Let $G$ be an arbitrary graph and let
  $T_m$ be a tree on $m$ vertices. Then
$$\hom(T_m,G)\leq \hom(S_m,G).$$
\end{Th}

We have already seen a few examples to the phenomenon that in many extremal
problems concerning trees it turns out that the maximal (minimal) value of the
examined parameter is attained at the star and the minimal (maximal) value is
attained at the path among trees on $n$ vertices (cf.~\cite{pc2,lp}). In what
follows we will show that this  phenomenon occurs quite frequently if one
studies homomorphisms of trees.

Let $Y_{a,b,c}$ be the starlike tree on $a+b+c+1$ vertices which has exactly
$3$ leaves and the vertex of degree $3$ has distance $a,b,c$ from the leaves,
respectively.

\begin{Th} \label{minimality-path} Let $T_n$ be a tree on $n$ vertices. Assume
  that for a tree $T_m$ we have
$$\hom(T_m,T_n)<\hom(T_m,P_n).$$
Then $T_n=Y_{1,1,n-3}$ and $n$ is even.
\end{Th}

In fact, we conjecture that we only have to exclude the case $n=4$ and
$T_4=S_4$. 

\begin{Conj}\label{disappointment}  Let $T_n$ be a tree on $n$ vertices, where
  $n\geq 5$. Then for any tree $T_m$ we have
$$\hom(T_m,P_n)\leq \hom(T_m,T_n).$$
\end{Conj}

An \emph{endomorphism} of a graph is a homomorphism from the graph to
itself. For a graph $G$, denote by $\End(G)$ the set of endomorphisms of
$G$. We remark that $\End(G)$ forms  a monoid  with respect to the composition
of mappings. 
One of the main results of this paper is the following extremal property about
the number of endomorphisms of trees. 

\begin{Th}\label{main} For all trees $T_n$ on $n$ vertices we have
$$|\End(P_n)|\leq|\End(T_n)|\leq|\End(S_n)|.$$
\end{Th}

This paper is organized as follows. In Section~\ref{treewalk}, we state the
tree-walk algorithm.  In Section~\ref{arb:graph}, we prove some lower bounds
involving Markov chains  and an upper bound (Theorem~\ref{of stars}) for the
number of homomorphisms from trees to an arbitrary graph.
Section~\ref{treewalks} is devoted to the proof of Theorem~\ref{of
  starlike}. The proofs of Theorem~\ref{minimality-path} and
Theorem~\ref{main} are given in Section~\ref{main:section}, where some lower
bounds concerning the homomorphisms of arbitrary trees are also proved.  In
Section~\ref{homtopath} we prove Theorem~\ref{into paths}, this part can be
read separately, it only builds on the algorithm of Section~\ref{treewalk}.

In order to make our paper transparent, we offer the following two tables,
Figure~\ref{table1} and \ref{table2}, which summarize our results. In both
tables, the first row follows from Theorem~\ref{th:1} or its generalization
Corollary~\ref{ex:totree}. The  last row is  obvious since $\hom(S_m,G)$ is
the sum of degree powers of $G$ and it also follows from
Corollary~\ref{ex:totree}. The first, second and third columns follow from
Theorem~\ref{into paths}, Theorem~\ref{of stars}  and Corollary~\ref{to stars}
respectively. The ``$\red{X}$" means that there is no inequality between the
two expressions in general and the ``\red{?}'' means that we don't know
whether the statement is true or not.   

\begin{figure}[h]
\begin{align*}
\hom(&P_n,P_n)\,\leq\,\hom(P_n,T_n)\,\leq\,\hom(P_n, S_n) \,\,\,\,\,\,\, \\
&\begin{turn}{90}
$\geq$
\end{turn}
\quad\quad\quad\quad\qquad\,\,
\red{?}
\quad\quad\quad\quad\qquad
\begin{turn}{90}
$\geq$
\end{turn}\\
\hom(&T_n,P_n)\,\leq\,\hom(T_n,T_n)\,\,\,\red{X}\,\,\hom(T_n, S_n)
\\
&\begin{turn}{90}
$\geq$
\end{turn}
\quad\quad\quad\quad\qquad
\begin{turn}{90}
$\geq$
\end{turn}
\quad\quad\quad\quad\qquad
\begin{turn}{90}
$\geq$
\end{turn}\\
\hom(&S_n,P_n)\,\leq\,\hom(S_n,T_n)\,\leq\,\hom(S_n, S_n) 
\end{align*}
\caption{\label{table1} The number of homomorphisms between trees of the same
  size.} 
\end{figure}

\begin{figure}[h]
\begin{align*}
\hom(&P_m,P_n)\,\leq\,\hom(P_m,T_n)\,\leq\,\hom(P_m, S_n) \\
&\begin{turn}{90}
$\geq$
\end{turn}
\quad\quad\quad\quad\qquad
\red{X}
\quad\quad\quad\quad\qquad
\begin{turn}{90}
$\geq$
\end{turn}\\
\hom(&T_m,P_n)\,\stackrel{(*)}{\leq}\,\hom(T_m,T_n)\,\,\,\red{X}\,\,\hom(T_m,
  S_n)\\ 
&\begin{turn}{90}
$\geq$
\end{turn}
\quad\quad\quad\quad\qquad
\begin{turn}{90}
$\geq$
\end{turn}
\quad\quad\quad\quad\qquad
\begin{turn}{90}
$\geq$
\end{turn}\\
\hom(&S_m,P_n)\,\leq\,\hom(S_m,T_n)\,\leq\,\hom(S_m, S_n) 
\end{align*}
\caption{\label{table2} The number of homomorphisms between trees of different
sizes. The $(*)$ means that there are some well-determined (possible)
counterexamples which should be excluded.}
\end{figure}

\section{The Tree-walk algorithm}
\label{treewalk}

In this section we shall state an algorithm for the number of homomorphisms
from a tree to any graph by the Transfer-matrix method. As a generalized
concept of walks in graphs, we call a homomorphism from a tree to a graph a
\emph{tree-walk} on this graph.

 Let ${\bf a}=(a_{1}, a_{2}, \dots, a_{n})$ and
${\bf b}=(b_{1}, b_{2}, \dots, b_{n})$ be two vectors. We usually denote by
$\lVert{\bf a}\rVert=a_{1}+a_{2}+\dots+a_{n}$ the \emph{norm} of ${\bf a}$ and
by ${\bf a}\ast{\bf b}=(a_{1}b_{1}, \dots, a_{n}b_{n})$ the \emph{Hadamard
  product} of ${\bf a}$ and ${\bf b}$.  
Denote by ${\bf 1_{n}}$ the $n$-dimensional row vector with all entries are
equal to $1$. Let $G$ be a graph with $n$ vertices. The \emph{adjacency
  matrix} of $G$ is the $n\times n$ matrix $A_G :=(a_{uv})_{u,v\in V(G)}$,
where $a_{uv}=1$ when $uv\in E(G)$, otherwise $0$. We begin with a fundamental
lemma about the number of walks in a graph. 

\begin{Lemma} \label{le:1}
Let $G$ be a labeled graph and $A=A_G$ the adjacency matrix of $G$. Then the
$(i,j)$-entry of the matrix $A^{n}$ counts the number of walks in $G$ from
vertex $i$ to vertex $j$ with length $n$. 
\end{Lemma}

\begin{proof}
By easy induction on $n$. See for example \cite[Theorem 4.7.1]{st}.
\end{proof}

\begin{Def}[hom-vector]
Let $T$ be a tree and $G$ be a  graph with vertices labeled by $1, 2, \dots,
n$. Let $v\in V(T)$ be any vertex of $T$. The $n$-dimensional vector 
$${\bf
  h}(T,v,G):=(h_{1}, h_{2}, \dots, h_{n})$$ where 
$$
h_{i}=|\{f \in \Hom(T,G) \ | \ f(v)=i\}|,
$$
is called the \emph{hom-vector} at $v$ from $T$ to $G$. Clearly,
$\hom(T, G) = \lVert{\bf h}(T,v,G)\rVert$. 
\end{Def}
The following Tree-walk algorithm
can be viewed as a generalization of Lemma~\ref{le:1} for computing the number
of tree-walks in graphs. 
\vskip 0.1in

\noindent {\bf The Tree-walk algorithm.}
Let $A=A_G$ be the adjacency matrix of the labeled graph $G$. Let $v$ be a
leaf of the tree $T$. We now give the algorithm to compute ${\bf h}(T,v,G)$.  

Starting from a leaf of the tree $T$ other than $v$ (a tree usually has at
least two leaves), $v_{0}$ say. Walking along the tree until you come to the
first vertex with degree greater than two, $v_{1}$ say. Let $T_{1}$ denote the
path from $v_{0}$ to $v_{1}$ with length $d_{1}$. By Lemma~\ref{le:1}, we have
${\bf h}(T_{1},v_{1},G)={\bf 1_{n}} A^{d_{1}}$. If the degree of vertex
$v_{1}$ is $k$, we denote its $k-1$ branches other than $T_{1}$ by
$T_{2},T_{3},\dots,T_{k}$. There is one branch that contains the vertex $v$,
$T_{k}$ say. We can compute the hom-vectors ${\bf h}(T_{2},v_{1},G), {\bf
  h}(T_{3},v_{1},G), \dots , {\bf h}(T_{k-1},v_{1},G)$ by recursively using
this algorithm. Clearly we have  
\begin{equation}\label{eq:al1}
{\bf h}(\cup_{i=1}^{k-1}T_{i},v_{1},G)={\bf h}(T_{1},v_{1},G)\ast{\bf
  h}(T_{2},v_{1},G)\ast \dots \ast{\bf h}(T_{k-1},v_{1},G), 
\end{equation}
where $\cup_{i=1}^{k-1}T_{i}$ is the union of $T_{i}$, $1\leq i\leq k-1$.
Now continue to walk on the tree $T$ from $v_{1}$ along the branch $T_{k}$
until you go to the first vertex with degree bigger than two, $v_{2}$ say. Let
$d_2$ be the length of the path, denoted by $P$, from $v_1$ to $v_2$. By
induction on $d_2$, one easily shows that 
\begin{equation}\label{eq:al2}
{\bf h}(\cup_{i=1}^{k-1}T_{i}\cup P,v_{2},G)={\bf
  h}(\cup_{i=1}^{k-1}T_{i},v_{1},G)A^{d_{2}}, 
\end{equation}
which is the most important operation of this algorithm. Do the same thing on
$v_{2}$ as what we have done on $v_{1}$, and the process stop until we walk to
the last vertex of $T$, which  must be $v$ according to the algorithm.  At the
same time, the hom-vector at $v$ from $T$ to $G$ is obtained and thus
$\hom(T,G)$.  

Note that we can modify this algorithm to obtain ${\bf h}(T,v,G)$ at any vertex
$v\in V(T)$ by using operation~\eqref{eq:al1}.  
The following typical example would explain our algorithm better.

\begin{figure}[h!]
\begin{center}
\scalebox{.65}{\includegraphics{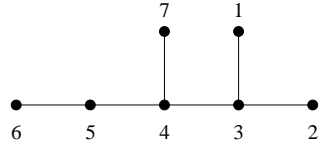}}
\end{center}
\caption{  \label{t1}A labeled tree on
  $7$ vertices.}     
\end{figure}

\begin{Exa} Let $T$ be the tree in Fig.~\ref{t1}. We also
choose $G$ to be $T$. Denote by $T[V]$ the induced subtree on the vertex set
$V\subseteq V(T)$. To compute ${\bf h}(T,7,T)$ by Tree-walk algorithm, we
start at vertex $6$ and first stop at $4$ which is a vertex with degree
greater than two. By Lemma~\ref{le:1} we have  
$${\bf h}(T[6,5,4],4,T)={\bf 1_{7}} A_T^{2}=(3,3,5,6,4,2,3).$$
 Using the algorithm recursively, to compute ${\bf h}(T[1,2,3,4],4,T)$, we
start at vertex $1$ and first stop at $3$. By operation~\eqref{eq:al1}, we
have  
\begin{align*}
{\bf h}(T[1,2,3],3,T)=&{\bf h}(T[1,3],3,T)\ast{\bf h}(T[2,3],3,T)\\
=&({\bf 1_{7}} A_T)\ast({\bf 1_{7}} A_T)
=(1,1,9,9,4,1,1). 
\end{align*}
We continue to walk from vertex $3$ and stop at $4$, by
operation~\eqref{eq:al2} we get  
$$
{\bf h}(T[1,2,3,4],4,T)={\bf h}(T[1,2,3],3,T)A_T=(9,9,11,14,10,4,9). 
$$
Now again by operation \eqref{eq:al1} we have 
\begin{align*}
{\bf h}(T[1,2,3,4,5,6],4,T)=&{\bf h}(T[1,2,3,4],4,T)\ast{\bf h}(T[6,5,4],4,T)\\
=&(27,27,55,84,40,8,27).
\end{align*}
To finish the computing, we continue to walk from $4$ and at last stop at $7$
and  apply~\eqref{eq:al2} again we get   
\begin{align*}
{\bf h}(T,7,T)={\bf h}(T[1,2,3,4,5,6],4,T)A_T
=(55,55,138,122,92,40,84).
\end{align*}
Thus $\hom(T,T)=\lVert{\bf h}(T,7,T)\Vert=586.$
\end{Exa}

\noindent {\bf Alternative way: recursions.} There is an alternative way to
think to the Tree-walk algorithm, namely we consider two type of recursion
steps. 
\bigskip

\textbf{Recursion 1.} If we have a tree $T$ with a non-leaf root vertex $v$,
then we can decompose $T$ to $T_1\cup T_2$ such that $V(T_1)\cap
V(T_2)=\{v\}$, and $T_1$ and $T_2$ are strictly smaller than $T$.  In this case
$${\bf h}(T,v,G)={\bf h}(T_1,v,G)\ast {\bf h}(T_2,v,G).$$
\bigskip

\textbf{Recursion 2.} If we have a tree $T$ with a root vertex $v$ which is a
leaf with the unique neighbor $u$ then
$${\bf h}(T,v,G)={\bf h}(T-v,u,G)A,$$
where $A$ is the adjacency matrix of $G$.
\bigskip

Hence we use Recursion 1 or Recursion 2 according to the root vertex $v$ is a
non-leaf or a leaf. In most of the proofs we simply check whether some
property of the vector ${\bf h}(T,v,G)$ remains valid after applying Recursion
1 and Recursion 2. 

\section{Graph homomorphisms from trees}
\label{arb:graph}

\subsection{Markov chains and homomorphisms}

\begin{Th} \label{markov}
Let $G$ be a graph and let $P=(p_{ij})$ be a Markov chain on $G$:
$$\sum_{j\in N(i)}p_{ij}=1\ \ \ \ \mbox{for all}\ \ i\in V(G),$$
where $p_{ij}\geq 0$ and $p_{ij}=0$ if $(i,j)\notin E(G)$.
Let $Q=(q_i)$ be the stationary distribution of $P$:
$$\sum_{j\in N(i)}q_jp_{ji}=q_i \ \ \ \ \mbox{for all}\ \ i\in V(G).$$
Let us define the following entropies:
$$H(Q)=\sum_{i\in V(G)}q_i\log \frac{1}{q_i},$$
and
$$H(D|Q)=\sum_{i\in V(G)}q_i\log d_i,$$
where $d_i$ is the degree of the vertex $i$, and let
$$H(P|Q)=\sum_{i\in V(G)}q_i\left(\sum_{j\in N(i)}p_{ij}\log
\frac{1}{p_{ij}}\right).$$ 
Let $T_m$ be a tree with $\ell$ leaves on $m$ vertices, where $m\geq 3$. Then
$$\hom(T_m,G)\geq \exp\biggl(H(Q)+\ell H(D|Q)+(m-1-\ell)H(P|Q)\biggr).$$
\end{Th}

\begin{proof} Let $v$ be  a root of $T$.
Let $a_i$ be the number of homomorphisms of $T_m$ into $G$ such that the root
vertex $v$ goes into the vertex $i\in V(G)$. Let  
$$F(T_m(v),G)=\prod_{i=1}^na_i^{q_i}.$$

We will show by induction on $m$ that
$$F(T_m(v),G)\geq \exp \left(\ell^* H(D|Q)+(m-1-\ell^*)H(P|Q)\right),$$
where $\ell^*$ is the number of leaves different from $v$, so it is $\ell$ if
$v$ is not a leaf and $\ell-1$ if $v$ is  a leaf. Note that 
$$F(K_2(v),G)=\exp H(D|Q).$$

If $v$ is not a leaf of $T_m$, then we can decompose $T_m$ to $T_1(v)$ and
$T_2(v)$. 
Then
$$F(T_m(v),G)=F(T_1(v),G)F(T_2(v),G)$$
because of the Hadamard-products of the hom-vectors.
From this the claim follows immediately by induction.

If $v$ is a leaf of $T_m$ with the unique neighbor $u$, then let
$${\bf h}(T_m-v,u,G)=(b_1,\dots ,b_n).$$
So $a_i=\sum_{j\in N(i)}b_j$.

For positive numbers $r_1,\dots ,r_t$ and positive weights $w_1,\dots ,w_t$
with $\sum_{i=1}^tw_i=1$, the weighted AM-GM inequality says that
$$r_1+\dots+r_t=w_1\left(\frac{r_1}{w_1}\right)+\dots
+w_t\left(\frac{r_t}{w_t}\right)\geq \left(\frac{r_1}{w_1}\right)^{w_1}\dots 
\left(\frac{r_t}{w_t}\right)^{w_t}=$$ 
$$=\exp \left(\sum_{i=1}^tw_i\log
\frac{1}{w_{i}}\right)\prod_{i=1}^tr_i^{w_i}.$$ 
Hence
$$F(T_m(v),G)=\prod_{i=1}^na_i^{q_i}=\prod_{i=1}^n\left(\sum_{j\in
  N(i)}b_j\right)^{q_i}\geq \prod_{i=1}^n\left(\prod_{j\in
  N(i)}\left(\frac{b_j}{p_{ij}}\right)^{p_{ij}}\right)^{q_i}=$$
$$=\prod_{i=1}^n\left(\prod_{j\in
  N(i)}\left(\frac{1}{p_{ij}}\right)^{p_{ij}q_i}\right)\prod_{i=1}^nb_i^{\sum_{j\in
    N(i)}p_{ji}q_j}=\prod_{i=1}^n\left(\prod_{j\in
  N(i)}\left(\frac{1}{p_{ij}}\right)^{p_{ij}q_i}\right)\prod_{i=1}^nb_i^{q_i}$$
In the last step we used that $Q$ is a stationary distribution with respect to
$P$. Hence 
$$F(T_m(v),G)\geq \exp(H(P|Q))F((T_m-v)(u),G).$$
Now the claim follows by induction.

To finish the proof of the theorem, we only have to choose a nonleaf root and
use that 
$$\hom(T_m,G)=\sum_{i=1}^na_i\geq \exp(H(Q))F(T_m(v),G).$$
\end{proof}

\begin{Rem} Note that the inequality $H(D|Q)\geq H(P|Q)$ always
  hold. Consequently, 
$$\hom(T_m,G)\geq \exp(H(Q)+(m-1)H(P|Q)).$$
As Theorem~\ref{markov} suggests, this is an inequality for entropies and
indeed, it can be proved by this way. By $P$ and $Q$, we defined a
distribution on the set of homomorphisms: we choose a root according to $Q$,
then we choose every nonleaf new vertex according to $P$ and finally we choose
the leaves uniformly. The entropy of this distribution is
exactly $H(Q)+\ell H(D|Q)+(m-1-\ell)H(P|Q)$ since every nonleaf vertex has
distribution $Q$. Note that this entropy is smaller than the entropy
of the uniform distribution, that is, $\log \hom(T_m,G)$. For basic facts about entropy, see for example~\cite{ct}. 
\end{Rem}

\begin{Th} \label{spectral-markov} 
Let $G$ be a connected graph on the vertex set 
$\{1,2,\dots ,n\}$ and let  $\lambda$ be the largest eigenvalue of the
adjacency matrix of the graph $G$. Let $\underline{y}$ be a positive
eigenvector of unit length corresponding to $\lambda$. Let $q_i=y_i^2$. Then
for any rooted tree $T_m$ on $m$ vertices we have
$$\hom(T_m,G)\geq \exp(H_{\lambda}(G))\lambda^{m-1},$$
where 
$$H_{\lambda}(G)=\sum_{i=1}^nq_i\log \frac{1}{q_i}$$ 
is the spectral entropy of the graph $G$.
\end{Th}

\begin{proof} We will use Theorem~\ref{markov}. Let $p_{ij}=\frac{y_j}{\lambda
    y_i}$. Since $\underline{y}$ is a positive eigenvector, we have $p_{ij}>0$. 
For all $i$ we have $\lambda y_i=\sum_{j\in N(i)}y_j$, thus $\sum_{j\in
  N(i)}p_{ij}=1$. For $q_i=y_i^2$ we have 
$$q_ip_{ij}=y_i^2\frac{y_j}{\lambda y_i}=\frac{1}{\lambda}y_iy_j=q_jp_{ji}.$$
Hence
$$\sum_{i\in N(j)}q_jp_{ji}=\sum_{i\in N(j)}q_ip_{ij}=q_i.$$
This means that $P=(p_{ij})$ is a Markov chain with stationary distribution
$Q=(q_i)$. The conditional entropy
$$H(P|Q)=\sum_{i\in V(G)}q_i\left(\sum_{j\in N(i)}p_{ij}\log
\frac{1}{p_{ij}}\right)= 
\sum_{i\in V(G)}y_i^2\left(\sum_{j\in N(i)}\frac{y_j}{\lambda y_i}\log
\frac{\lambda y_i}{y_j}\right)=$$
$$=\sum_{\{i,j\}\in E(G)} \frac{y_iy_j}{\lambda}\left(2\log \lambda+\log
\frac{y_i}{y_j}+\log \frac{y_j}{y_i}\right)=\log (\lambda)
\frac{1}{\lambda}\sum_{(i,j)\in E(G)}y_iy_j=\log \lambda.$$
Hence the result follows from Theorem~\ref{markov}.
\end{proof}

\begin{Rem} A Markov chain is called reversible if $q_ip_{ij}=q_jp_{ji}$ for
  all $i,j\in V(G)$. As we have seen, the Markov chain constructed in the
  previous proof is reversible. It is not hard to show that on trees every
  Markov chains are reversible.
\end{Rem}

\begin{Rem}
Theorem~\ref{spectral-markov} is the best possible in the sense that there
cannot be a larger number than $\lambda$ in such a statement since
$$\hom(P_m,G)\leq n\lambda^{m-1}.$$
Indeed, 
$$\frac{\hom(P_m,G)}{n}=\frac{{\bf 1_{n}}^TA^{m-1}{\bf 1_{n}}}{{\bf 1_{n}}^T{\bf
    1_{n}}} \leq \max_{v\neq
    \underline{0}}\frac{v^TA^{m-1}v}{v^Tv}=\lambda_{\max}(A^{m-1})=\lambda^{m-1}.$$

Note that we can deduce that if $(T_m)_{m=1}^{\infty}$ is a sequence of trees
such that $T_m$ has $m$ vertices then
$$\liminf_{m\to \infty}\hom(T_m,G)^{1/m}\geq
\liminf_{m\to \infty}\hom(P_m,G)^{1/m}=\lambda.$$
This result could have been deduced as well from a theorem of B. Rossman and
E. Vee~\cite{br} claiming that
$$\hom(T_m,G)\geq \hom(C_m,G),$$
where $C_m$ is the cycle on $m$ vertices. In fact, this was proved for
directed trees and cycles, but it implies the inequality for undirected tree
and cycle. This result can also be deduced from Theorem 3.1 of \cite{sk}. 
\end{Rem}

\begin{Th} \label{degree-markov}
Let $G$ be a graph on the vertex set $\{1,2,\dots ,n\}$ with
$e$ edges and with degree sequence $(d_1,\dots ,d_n)$. Then for any tree $T_m$
on $m$ vertices we have 
$$\hom(T_m,G)\geq 2e\cdot  C^{m-2},$$
where
$$C=\left(\prod_{i=1}^nd_i^{d_i}\right)^{1/2e}.$$
\end{Th}

\begin{proof} Let us consider the following classical Markov chain:
  $p_{ij}=\frac{1}{d_i}$ if $j\in N(i)$. The stationary distribution is
  $q_i=\frac{d_i}{2e}$. Note that
$$H(P|Q)= \sum_{i\in V(G)}q_i\left(\sum_{j\in N(i)}p_{ij}\log
  \frac{1}{p_{ij}}\right)= 
\sum_{i\in V(G)}q_i\log d_i=\frac{1}{2e}\sum_{i\in V(G)}d_i\log d_i=\log C$$
and 
$$H(Q)+H(P|Q)=\sum_{(i,j)\in E(G)}q_ip_{ij}\log
\frac{1}{q_ip_{ij}}=\sum_{(i,j)\in E(G)}\frac{1}{2e}\log (2e)=\log (2e).$$
Hence the result follows from Theorem~\ref{markov}.
\end{proof}

\begin{Def} The homomorphism density $t(H,G)$ is defined as follows:
$$t(H,G)=\frac{\hom(H,G)}{|V(G)|^{|V(H)|}}.$$
This is the probability that a random map is a homomorphism.
\end{Def}

Sidorenko's conjecture says that
$$t(H,G)\geq t(K_2,G)^{e(H)}$$
for every bipartite graph $H$ with $e(H)$ edges. It is known that Sidorenko's
conjecture~\cite{si2} is true for trees. By now, there are many proofs for
this particular case of Sidorenko's conjecture: see \cite{cj,xl} and it can be
deduced as well from Theorem 3.1 of~\cite{sk}. Below we give a new proof for
this fact. 

\begin{Th} For any tree $T_m$ on $m$ vertices and a graph $G$ we have
$$t(T_m,G)\geq t(K_2,G)^{m-1}.$$
\end{Th}

\begin{proof} Let $|V(G)|=n$. The theorem will immediately follows form
  Theorem~\ref{degree-markov}. By convexity of the function $x\log x$ we have
$$\frac{1}{2e}\sum_{i\in V(G)}d_i\log d_i\geq
  \frac{1}{2e}n\left(\frac{2e}{n}\log \frac{2e}{n}\right)=\log \frac{2e}{n}.$$
Hence
$$t(T_m,G)=\frac{\hom(T_m,G)}{n^{m}}\geq
\frac{1}{n^{m}}2e\left(\frac{2e}{n}\right)^{m-2}=\left(\frac{2e}{n^2}\right)^{m-1}=t(K_2,G)^{m-1}.$$      
\end{proof}

\subsection{Sidorenko's theorem on extremality of stars}
The objective of this section is to give a new proof for Theorem~\ref{of
stars} in order to keep this paper self-contained. This was proved originally
by Sidorenko \cite{si}. Our proof is very similar to the original one, but it
is slightly more elementary. 

Before we start the proof we will need two definitions and two lemmas.

\begin{Def} Let $M_u$ and $N_v$ be two rooted graphs with root vertices $u$
  and $v$, respectively. Then $M_u\circ_{u=v}N_v$ denotes the graph
  obtained from $M_u\cup N_v$ by identifying the vertices $u$ and $v$.
\end{Def}   

\begin{Lemma} Let $R_{u,v}$ be a graph with specified (not necessarily
distinct) vertices $u$ and $v$. Let $J_{u'}$ and $K_{v'}$ be two graphs with
root vertices $u'$ and $v'$. Finally, let the graphs $A,B$ and $C$ be obtained
from $R_{u,v},J_{u'},K_{v'}$ as follows:
$$A=(R_{u,v}\circ_{u=u'}J_{u'})\circ_{v=v'}K_{v'},$$
$$B=(R_{u,v}\circ_{u=u'}J_{u'})\circ_{u=u'}J_{u'},$$
$$C=(R_{u,v}\circ_{v=v'}K_{v'})\circ_{v=v'}K_{v'}.$$
(In other words, in $B$ and $C$ we attach two copies of the same graph at the
specified vertex.) Then for any graph $G$ we have
$$2\hom(A,G)\leq \hom(B,G)+\hom(C,G).$$
\end{Lemma}

\begin{proof} Let $i,j\in V(G)$ and let $h(R_{u,v},i,j)$ denote the number of
  homomorphisms of $R_{u,v}$ to $G$ where $u$ goes to $i$ and $v$ goes to $j$.
We similarly define  $h(J_{u'},i)$ and $h(K_{v'},j)$. Then
$$\hom(A,G)=\sum_{i,j\in V(G)}h(R_{u,v},i,j)h(J_{u'},i)h(K_{v'},j).$$
Similarly,
$$\hom(B,G)=\sum_{i,j\in V(G)}h(R_{u,v},i,j)h(J_{u'},i)^2,$$
and 
$$\hom(C,G)=\sum_{i,j\in V(G)}h(R_{u,v},i,j)h(K_{v'},j)^2.$$
Hence
\begin{align*}
&\hom(B,G)+\hom(C,G)-2\hom(A,G)\\
=&\sum_{i,j\in V(G)}
h(R_{u,v},i,j)(h(J_{u'},i)-h(K_{v'},j))^2\geq 0.
\end{align*}
We are done.
\end{proof}

\begin{Def} Let $d(u,v)$ be the distance of the vertices $u,v\in V(G)$. Then the
  {\em Wiener-index} $W(G)$ of a graph $G$ is defined as
$$W(G):=\sum_{u,v\in V(G)}d(u,v).$$
\end{Def} 

In our application $R_{u,v}$ will be a tree and $J_{u'}$ and $K_{v'}$ be the
trees on $2$ vertices. 
The following lemma about the Wiener-index is trivial.

\begin{Lemma} Let $R_{u,v}$ be a tree with distinct vertices $u$ and $v$.  Let
  $J_{u'}$ and $K_{v'}$ be two copies of the two-node trees with root vertices
  $u'$ and $v'$, respectively. Finally, let the graphs $A,B$ and $C$ be obtained
  from $R_{u,v},J_{u'},K_{v'}$ as in the former lemma.  
Then $2W(A)>W(B)+W(C)$.
\end{Lemma}

\begin{proof}[Proof of Theorem~\ref{of stars}]
Let $\mathcal{T}_G$ be the set of those trees $F$ on $m$ vertices for which
$\hom(F,G)$ is maximal. Let $T\in \mathcal{T}_G$ be the tree for which $W(T)$
is minimal. We show that $T=S_m$. Assume for contradiction that $T\neq
S_n$. Then $T$ has two leaves, $a$ and $b$ such that $d(a,b)\geq 3$. Let $u$
and $v$ be the unique neighbors of $a$ and $b$, respectively. Then $u\neq
v$. Let $R_{u,v}=T-\{a,b\}$, $J_{u'}=\{u',a\}$ and $K_{v'}=\{v',b\}$. Then
$$A=(R_{u,v}\circ_{u=u'}J_{u'})\circ_{v=v'}K_{v'}=T.$$
As in the lemmas, let
$$B=(R_{u,v}\circ_{u=u'}J_{u'})\circ_{u=u'}J_{u'},$$
$$C=(R_{u,v}\circ_{v=v'}K_{v'})\circ_{v=v'}K_{v'}.$$
Note that $B$ and $C$ are also trees on $m$ vertices. By the Lemma we have
$$2\hom(A,G)\leq \hom(B,G)+\hom(C,G).$$
Since $A=T\in \mathcal{T}_G$, then $\hom(B,G)+\hom(C,G)\leq 2\hom(A,G)$.
So $\hom(A,G)=\hom(B,G)=\hom(C,G)$ implying that $B,C\in \mathcal{T}_G$ as
well.   But then $2W(T)>W(B)+W(C)$, so one of them has strictly smaller
Wiener-index than $T$, this contradicts the choice of $T$. 
Hence $T$ must be $S_m$.
\end{proof}

\begin{Rem} After all, it is a natural question whether it is true or not that
$$\hom(P_m,G)\leq \hom(T_m,G)$$
for any tree $T_m$ on $m$ vertices. Surprisingly, the answer is no! It was
already known to A. Leontovich \cite{al}. It turns out that even if restrict
$G$ to be tree there is a counterexample. Let $E_7$
be the tree obtained from $P_6$ by putting a pendant edge to the third vertex
of the path. Then there is a tree $T$ for which
$$\hom(P_7,T)> \hom(E_7,T).$$
The following tree $T$ is suitable: let $T=T(k_1,k_2,k_3)$ be the tree where
the root vertex $v_0$ have $k_1$ neighbors, all of its neighbors has $k_2+1$
neighbors and the vertices having distance $2$ from $v_0$ have $k_3+1$
neighbors. If we choose $k_1,k_2,k_3$ such that $k_2\ll k_1\ll k_3\ll k_1k_2$
(for instance $k_i=k^{\alpha_i}$, where
$\alpha_2<\alpha_1<\alpha_3<\alpha_1+\alpha_2$ and  $k$ is large), then  
$$\hom(P_7,T)-\hom(E_7,T)=k_1^2k_2^2k_3^2+o(k_1^2k_2^2k_3^2).$$
\end{Rem}

\section{Tree-walks on trees}
\label{treewalks}
The main purpose of this section is to prove Theorem~\ref{of starlike}. We
shall give an inductive proof of Theorem~\ref{th:1} which can be 
generalized to tree-walks by tree-walk algorithm.

We first need  some notations. 
Let $T$ be a tree and $T'=KC(T,p_{0},p_{k})$ its KC-transformation with
respect to a path $P$ of length $k$, a path with vertices labeled
consecutively with $p_0, p_1, \ldots, p_k$. We denote by $A$ and $B$ the
components of $p_{0}$ and $p_{k}$ in the subgraph of $T$ by deleting all the
edges of $P$. Let $A'$, $B'$ and $P'$ be the components of $T'$ corresponding
with components $A$, $B$ and $P$ under the KC-transformation,
respectively. The vertices of the path $P'$ will be labeled consecutively
with $p'_0, p'_1,\dots, p'_k$, where $p'_i$ is corresponding to $p_i$ for
$0\leq i\leq k$. So $p_0\in A, p_k\in B$ in $T$, and $p'_0\in A',B'$ in
$T'$.

\begin{Lemma} \label{funny} Let $a_1,a_2,b_1,b_2,c_1,c_2,d_1,d_2$ be positive
numbers satisfying the inequalities: $a_i\geq \max(c_i,d_i)$,  $a_i+b_i\geq
c_i+d_i$  for $i=1,2$. Then $a_1a_2\geq \max(c_1c_2,d_1d_2)$ and
$a_1a_2+b_1b_2\geq c_1c_2+d_1d_2$. 
\end{Lemma}

\begin{proof} Clearly, we only have to prove that $a_1a_2+b_1b_2\geq
  c_1c_2+d_1d_2$, the other inequality is trivial. 
Note that $b_i\geq \max(0,c_i+d_i-a_i)$. If one of $c_i+d_i-a_i<0$, say
$c_1+d_1<a_1$  then 
$$a_1a_2\geq (c_1+d_1)a_2\geq c_1c_2+d_1d_2.$$
If both $c_i+d_i-a_i\geq 0$ for $i=1,2$, then
$$a_1a_2+b_1b_2\geq a_1a_2+(c_1+d_1-a_1)(c_2+d_2-a_2)=$$
$$=c_1c_2+d_1d_2+(a_1-c_1)(a_2-d_2)+(a_1-d_1)(a_2-c_2)\geq c_1c_2+d_1d_2.$$
Hence we are done.
\end{proof}

We first treat the case of $k$ being even. 

\begin{proof}[ Proof of first part of Theorem~\ref{of starlike}] 
In this proof $k$ is even: $k=2t$. We label
$V(A)\setminus p_0$ with $\{a_m\ |\ 1\leq m\leq M\}$, $V(A')\setminus p'_0$
with $\{a'_i\ |\ 1\leq m\leq M\}$, $V(B)\setminus p_k$ with $\{b_n\ |\ 1\leq
n\leq N\}$ and $V(B)\setminus p'_0$ with $\{b'_n\ |\ 1\leq n\leq N\}$, where
$a_m$ (resp. $b_n$) is corresponding to $a'_m$ (resp. $b'_n$) under the
KC-transformation. For  $v\in H$, we always write  
$$ {\bf h}(H,v,T)=(a_{1}, a_{2}, \dots, a_{M}, p_{0}, p_{1}, \dots, p_{k}, b_{1},
b_{2}, \dots, b_{N}) $$
and
$$ {\bf h}(H,v,T')=(a'_{1}, a'_{2}, \dots, a'_{M}, p'_{0}, p'_{1}, \dots,
p'_{k}, b'_{1}, b'_{2}, \dots, b'_{N}), $$
where we use the labels of vertices of $T$ and $T'$ to index the parameters of
the hom-vectors to $T$ and $T'$ respectively. We hope that it will not cause
any confusion. We shall prove by induction on the steps of tree-walk algorithm
that  
\begin{align}
a'_m\geq a_m,\  b'_n\geq b_n\label{ine1}\\
p'_{i}+p'_{k-i}\geq p_{i}+p_{k-i}\label{ine2}\\ 
p'_{i}\geq p_{i},\  p'_{i}\geq p_{k-i}\label{ine3} 
\end{align}
for $1\leq m \leq M, 1\leq n \leq N, 0\leq i \leq t$.
We simply write ${\bf h}(H,v,T)\leq {\bf h}(H,v,T')$ if two such hom-vectors
from $H$ to $T$ and $T'$ satisfy the  inequalities~\eqref{ine1},~\eqref{ine2}
and~\eqref{ine3}.  

It is easy to verify that all these inequalities are satisfied after applying
any recursion step of the tree-walk algorithm. When $v$ is a leaf of $H$ then
it is trivial that these inequalities are preserved. If $v$ is not a leaf then
we use Lemma~\ref{funny} to see that the Hadamard-product preserves these
inequalities.   
\end{proof}

\begin{Lemma} \label{path-odd k}
Let $k$ be odd and assume that $A$ and $B$ have at least two
vertices.  Let $a_m(r),b_n(r),p_i(r),a'_m(r),b'_n(r),p_i'(r)$ denote the
number of homomorphism of $P_r$ into $T$ and $T'$, respectively, such that
the endvertex of $P_r$ goes to the vertices
$a_m,b_n,p_i,a'_m,b'_n$ and $p_i'$, respectively. Then the following
inequalities hold for every  $r$:  
\begin{align}  
a'_m(r)\geq a_m(r),\  b'_n(r)\geq b_n(r)\\
p'_i(r)+p'_j(r)\geq p_i(r)+p_j(r)   \label{ine4}        \\ 
p'_i(r)+p'_j(r)\geq p_{k-i}(r)+p_{k-j}(r) \label{ine5}
\end{align}
for $1\leq m \leq M, 1\leq n \leq N$ and $i+j\leq k$.
\end{Lemma}

\begin{proof} We prove the claim by induction on $r$. For $r=1,2$, the claim
  is trivial. Note that we only have to prove that 
\begin{align*}  
a'_m(r)\geq a_m(r),\  b'_n(r)\geq b_n(r) \\
p'_i(r)+p'_j(r)\geq p_i(r)+p_j(r) 
\end{align*}
for $1\leq m \leq M, 1\leq n \leq N$ and $i+j\leq k$.
We obtain the inequality 
$$p'_i(r)+p'_j(r)\geq p_{k-i}(r)+p_{k-j}(r)$$ 
by simply exchanging the role of $A$ and $B$. Also note that if we put $i=j$
in the inequality~\eqref{ine4} and~\eqref{ine5} we obtain that $p'_i(r)\geq
p_i(r), p_{k-i}(r)$ for $i<k/2$.   

Observe that for any vertex $v$ we have:
$$v(r)=\sum_{u\in N(v)}u(r-1).$$

We will treat the cases $k=1$ and $k\geq 3$ separately. 
\bigskip

\noindent {\bf Case 1: $k=1$.} In this case, we have to prove the inequalities:
$$a'_m(r)\geq a_m(r),\  b'_n(r)\geq b_n(r),p'_0(r)\geq \max(p_0(r),p_1(r)),
p'_0(r)+p'_1(r)\geq p_0(r)+p_1(r).$$

The inequalities $a'_m(r)\geq a_m(r),\  b'_n(r)\geq b_n(r)$ simply follow from
the inequalities $a'_m(r-1)\geq a_m(r-1),\  b'_n(r-1)\geq b_n(r-1)$, and 
$p'_0(r-1)\geq p_0(r-1),p_1(r-1)$.

Observe that
$$p'_0(r)=\sum_{a'_m\in N(p'_0)}a'_m(r-1)+\sum_{b'_n\in
  N(p'_0)}b'_n(r-1)+p'_1(r-1)=$$ 
$$= \sum_{a'_m\in N(p'_0)}a'_m(r-1)+\sum_{b'_n\in N(p'_0)}b'_n(r-1)+p'_0(r-2)\geq $$
$$\geq  \sum_{a_m\in N(p_0)}a_m(r-1)+\sum_{b_n\in N(p_1)}b_n(r-1)+p_0(r-2)\geq $$
$$\geq \sum_{a_m\in N(p_0)}a_m(r-1)+\sum_{b_n\in N(p_1)}b_n(r-2)+p_0(r-2)=$$
$$=\sum_{a_m\in N(p_0)}a_m(r-1)+p_1(r-1)=p_0(r).$$
We used the induction hypothesis and that $b_m(r-1)\geq b_m(r-2)$. In general,
$u(r)\geq u(r-1)$ since any homomorphism of $P_{r-1}$ starting at the vertex
$u$ can be extended to a homomorphism of $P_{r}$ starting at $u$. 
Clearly, we can get $p'_0(r)\geq p_1(r)$ similarly, or we just switch the role
of $A$ and $B$. 

Finally,
$$p'_0(r)+p'_1(r)=\sum_{a'_m\in N(p'_0)}a'_m(r-1)+\sum_{b'_n\in
  N(p'_0)}b'_n(r-1)+p'_0(r-1)+p'_1(r-1)\geq $$  
$$\geq \sum_{a_m\in N(p_0)}a_m(r-1)+\sum_{b_n\in
  N(p_1)}b_n(r-1)+p_0(r-1)+p_1(r-1)=p_0(r)+p_1(r).$$ 
Hence we are done in this case.
\bigskip

\noindent {\bf Case 2: $k\geq 3$.} Clearly, the  inequalities $a'_m(r)\geq
a_m(r),\  b'_n(r)\geq b_n(r)$ simply follow from the inequalities
$a'_m(r-1)\geq a_m(r-1),\  b'_n(r-1)\geq b_n(r-1)$, and  $p'_0(r-1)\geq
p_0(r-1),p_k(r-1)$ as before. 

So we only have to prove the inequality $p'_i(r)+p'_j(r)\geq p_i(r)+p_j(r)$
for $i+j\leq k$. We can assume that $i\leq j$. If $i\geq 1$, then $j\leq k-1$
and  
$$p'_i(r)+p'_j(r)=(p'_{i-1}(r-1)+p'_{j+1}(r-1))+(p'_{i+1}(r-1)+p'_{j-1}(r-1))\geq $$ 
$$\geq (p_{i-1}(r-1)+p_{j+1}(r-1))+(p_{i+1}(r-1)+p_{j-1}(r-1))=p_i(r)+p_j(r).$$  

So we only have to consider the case $i=0$. In this case we consider the cases
$j=0,j=1,2\leq j\leq k-2,j=k-1,j=k$ separately. Unfortunately, all of them
behaves a bit differently. 
\bigskip

\noindent {\bf Subcase $j=0$:}
\begin{align*}
2p'_0(r)&=2\biggl(\sum_{a'_m\in N(p'_0)}a'_m(r-1)+\sum_{b'_n\in
  N(p'_0)}b'_n(r-1)+p'_1(r-1)\biggr) \\
&\geq 2\biggl(\sum_{a_m\in N(p_0)}a_m(r-1)+p_1(r-1)\biggr)=2p_0(r),
\end{align*}
since $p'_1(r-1)\geq p_1(r-1)$, because $1<k/2$.
\bigskip

\noindent {\bf Subcase $j=1$:} 
$$
p'_0(r)+p'_1(r)=\sum_{a'_m\in N(p'_0)}a'_m(r-1)+\sum_{b'_n\in
  N(p'_0)}b'_n(r-1)+p'_1(r-1)+p'_0(r-1)+p'_2(r-1) 
$$
$$
\qquad\quad\geq \sum_{a_m\in N(p_0)}a_m(r-1)+p_1(r-1)+p_0(r-1)+p_2(r-1)=p_0(r)+p_1(r),
$$
since $p'_1(r-1)\geq p_1(r-1)$ and $p'_0(r-1)+p'_2(r-1)\geq p_0(r-1)+p_2(r-1)$.
\bigskip

\noindent {\bf Subcase $2\leq j\leq k-2$:}
Here we jump back from $r$ to $r-2$, so we need a few notations. Let $d_A$ and
$d_B$ denote the degree of $p'_0$ in $A$ and $B$, respectively. Furthermore,
let $d(v,u)$ denote the distance of the vertices $u$ and $v$.  
Then
$$p'_0(r)+p'_j(r)=\sum_{a'_m: d(a'_m,p'_0)=2}a'_m(r-2)+\sum_{b'_n: 
  d(b'_n,p'_0)=2}b'_n(r-2)+(d_A+d_B+1)p'_0(r-2)$$
$$\qquad\quad+p'_2(r-2)+p'_{j-2}(r-2)+2p'_{j}(r-2)+p'_{j+2}(r-2)\geq $$  
$$\quad\geq \sum_{a_m: d(a_m,p_0)=2}a_m(r-2)+(d_A+1)p_0(r-2)+p_2(r-2)+$$ 
$$\qquad\quad+p_{j-2}(r-2)+2p_{j}(r-2)+p_{j+2}(r-2)=p_0(r)+p_j(r),$$
since the inequality follows from the following inequalities:
\begin{align*}
a'_m(r-2)\geq a_m(r-2),\ b'_n(r-2)\geq 0 \\
(d_B-1)p'_0(r-2)\geq -p_0(r-2) \\
(d_A-1)p'_0(r-2)\geq (d_A-1)p_0(r-2) \\
p'_2(r-2)+p'_{j-2}(r-2)\geq p_2(r-2)+p_{j-2}(r-2) \\
2(p'_0(r-2)+p'_{j}(r-2))\geq 2(p_0(r-2)+p_{j}(r-2))\\
p'_0(r-2)+p'_{j+2}(r-2)\geq p_0(r-2)+p_{j+2}(r-2).
\end{align*}
\bigskip

\noindent {\bf Subcase $j=k-1$:}
$$p'_0(r)+p'_{k-1}(r)=\sum_{a'_m\in N(p'_0)}a'_m(r-1)+\sum_{b'_n\in
  N(p'_0)}b'_n(r-1)+p'_1(r-1)+p'_{k-2}(r-1)+p'_{k}(r-1)=$$
$$=\sum_{a_m: d(a'_m,p'_0)=2}a'_m(r-2)+d_Ap'_0(r-2)+\sum_{b'_n\in
  N(p'_0)}b'_n(r-1)+p'_0(r-2)+p'_2(r-2)+p'_{k-3}(r-2)+2p'_{k-1}(r-2).$$
On the other hand,
$$p_0(r)+p_{k-1}(r)=\sum_{a_m\in
  N(p_0)}a_m(r-1)+p_1(r-1)+p_{k-2}(r-1)+p_{k}(r-1)=$$
$$=\sum_{a_m:
  d(a_m,p_0)=2}a_m(r-2)+d_Ap_0(r-2)+p_0(r-2)+p_2(r-2)+p_{k-3}(r-2)+2p_{k-1}(r-2)+
\sum_{b_n\in N(p_k)}b_n(r-2).$$
The inequality $p'_0(r)+p'_{k-1}(r)\geq p_0(r)+p_{k-1}(r)$ follows from
\begin{align*}
a'_m(r-2)\geq a_m(r-2),\ b'_n(r-1)\geq b_n(r-1)\geq b_n(r-2)\\ 
(d_A-1)p'_0(r-2)\geq (d_A-1)p_0(r-2) \\
p'_2(r-2)+p'_{k-3}(r-2)\geq p_2(r-2)+p_{k-3}(r-2) \\
2(p'_0(r-2)+p'_{k-1}(r-2))\geq 2(p_0(r-2)+p_{k-1}(r-2)).
\end{align*}
\bigskip

\noindent {\bf Subcase $j=k$:}
$$p'_0(r)+p'_k(r)=\sum_{a'_m\in N(p'_0)}a'_m(r-1)+\sum_{b'_n\in
  N(p'_0)}b'_n(r-1)+p'_1(r-1)+p'_{k-1}(r-1)\geq $$ 
$$\geq \sum_{a_m\in N(p_0)}a_m(r-1)+\sum_{b_n\in
  N(p_k)}b_n(r-1)+p_1(r-1)+p_{k-1}(r-1)=p_0(r)+p_k(r).$$
\end{proof}

\begin{proof}[ Proof of the second part of Theorem~\ref{of starlike}]
From Lemma~\ref{path-odd k} we only keep the inequalities
\begin{align*}
a'_m(r)\geq a_m(r),\  b'_n(r)\geq b_n(r)\\
p'_i(r)\geq p_i(r), p_{k-i}(r)\\ 
p'_i(r)+p'_{k-i}(r)\geq p_{i}(r)+p_{k-i}(r)
\end{align*}
for $1\leq m \leq M, 1\leq n \leq N, 0\leq i \leq k/2$.

For a tree $H$ and $v\in H$, let us write  
$$ {\bf h}(H,v,T)=(a_{1}, a_{2}, \dots, a_{M}, p_{0}, p_{1}, \dots, p_{k}, b_{1},
b_{2}, \dots, b_{N}) $$
and
$$ {\bf h}(H,v,T')=(a'_{1}, a'_{2}, \dots, a'_{M}, p'_{0}, p'_{1}, \dots,
p'_{k}, b'_{1}, b'_{2}, \dots, b'_{N}),$$
where we use the labels of vertices of $T$ and $T'$ to index the parameters of
the hom-vectors to $T$ and $T'$, respectively. We say that ${\bf h}(H,v,T)\leq{\bf
  h}(H,v,T')$ if the following inequalities hold
\begin{align*}
a'_m\geq a_m,\  b'_n\geq b_n\\
p'_{i}+p'_{k-i}\geq p_{i}+p_{k-i}\\ 
p'_{i}\geq p_{i},\  p'_{i}\geq p_{k-i}
\end{align*}
for $1\leq m \leq M, 1\leq n \leq N, 0\leq i \leq k/2$.
As we have seen these inequalities hold for a path $P_r$ and its
endvertex. Since these inequalities are preserved for Hadamard-product by
Lemma~\ref{funny}, we see that ${\bf h}(H,v,T)\leq{\bf h}(H,v,T')$ for starlike
trees $H$, where $v$ is the center of the starlike tree. This implies that
$$\hom(H,T')\geq \hom(H,T).$$
\end{proof}

The following generalization of inequality~\eqref{path} follows immediately
from Proposition~\ref{pr1} and Theorem~\ref{of starlike}. 
\begin{Cor}\label{ex:totree}
Let $H$ be a  starlike tree and let $T_n$ be a tree on $n$ vertices. Then 
\begin{equation*}
\hom(H,P_n)\leq \hom(H,T_n)\leq \hom(H,S_n).
\end{equation*}
\end{Cor}

The reader may wonder that if inequality \eqref{eq:2} holds when $k$ is odd
and $H$ is not a starlike tree. This is not true in general. A counterexample
will be constructed in the following, which also shows that  
$$\hom(H,T_n)\leq \hom(H,S_n)$$
is not true for any tree $H$.

\begin{Pro} \label{bipartite} Let $T$ be a tree with color classes $A$ and $B$
considered as a bipartite graph. Then
$$\hom(T,S_n)=(n-1)^{|A|}+(n-1)^{|B|}.$$
\end{Pro}

\begin{Cor}\label{to stars}
 Let $T_m$ be a tree on $m$ vertices, then
$$\hom(P_m,S_n)\leq \hom(T_m,S_n)\leq \hom(S_m,S_n).$$
If $T\neq S_m$ then the second inequality is strict.
\end{Cor}

\begin{proof}[Proof of Proposition~\ref{bipartite}] Since $T$ and $S_n$ are
bipartite graphs, a color class of $T$ have to go into a color class of $S_n$.
If the color class $A$ goes to the center of $S_n$, then any vertex belonging
to the color class of $B$ can go to any leaf of the star, so it provides
$(n-1)^{|B|}$ homomorphisms. The other case provides $(n-1)^{|A|}$ homomorphisms.
\end{proof}

This simple proposition also shows us how  to construct a tree $T_n$ for which 
$\hom(T_n,T_n)>\hom(T_n,S_n)$. 

\begin{figure}[h!] 
\begin{center}
\scalebox{.65}{\includegraphics{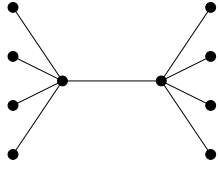}}\caption{The doublestar
  $S^*_{10}$.}  
\end{center}
\end{figure} 

Let $T_n=S^*_{2k}$ be the doublestar on $2k$ vertices with $2k-2$ leaves and
two vertices of degree $k$. Then it is easy to see that
$$\hom(S^*_{2k},S^*_{2k})>2(k-1)^{2(k-1)}=2(k^2-2k+1)^{k-1},$$
while 
$$\hom(S^*_{2k},S_{2k})=2(2k-1)^{k}.$$
Hence for $k\geq 5$ we have 
$$\hom(S^*_{2k},S^*_{2k})>\hom(S^*_{2k},S_{2k}).$$
Note that $S_{2k}$ can be obtained from $S^*_{2k}$ by a KC-transformation.

\section{Homomorphisms of arbitrary trees}
\label{main:section}

In this section we give the proof of Theorem~\ref{minimality-path} and
Theorem~\ref{main}. As we will see, Theorem~\ref{minimality-path} with some
additional observations implies Theorem~\ref{main}.

To prove Theorem~\ref{minimality-path} we will build on the fact that there
are not many homomorphisms into a path. Indeed, by Theorem~\ref{of stars} we
have 
$$\hom(T_m,P_n)\leq \hom(S_m,P_n)=(n-2)2^{m-1}+2.$$
So for a particular tree $T_n$, it is enough to prove that for every tree
$T_m$ we have
\begin{equation}\label{lower bound}
\hom(T_m,T_n)\geq (n-2)2^{m-1}+2.
\end{equation}
This would immediately imply that
\begin{equation}\label{path-minimality}
\hom(T_m,T_n)\geq \hom(T_m,P_n).
\end{equation}
We will prove that inequality~\ref{lower bound} is indeed true for all tree
$T_n$ with at least four leaves and for a large class of trees with three
leaves. For the remaining trees with three leaves we use Theorem~\ref{of
  starlike}.

\begin{figure}[h]
\begin{center}
\scalebox{.65}{\includegraphics{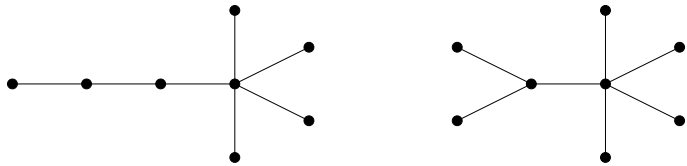}}\caption{The trees $T_8$
  (left) and $T'_8$ (right). \label{counter3}}    
\end{center}
\end{figure}

\begin{Rem} To prove Theorem~\ref{minimality-path} and Theorem~\ref{main} we
cannot rely entirely on the use of KC-transformation. That is why we had to
find another strategy to prove these theorems.

Indeed, KC-transformation does not always increase the
number of endomorphisms of trees. The first counterexample is the two trees
on $8$ vertices in Fig.~\ref{counter3}. The tree $T'_8$ is the
KC-transformation of $T_8$, but  
$|\End(T'_8)|=10430<17190=|\End(T_8)|.$ 
\end{Rem}

\subsection{The extremality of star}

Note that Theorem~\ref{of stars} and Theorem~\ref{of starlike} together
implies the following chain of inequalities:
$$|\End(T_n)|=\hom(T_n,T_n)\leq \hom(S_n,T_n)\leq \hom(S_n,S_n)=|\End(S_n)|,$$
since $S_n$ is a starlike tree. In this section, we will also give a direct
proof for it. 
\bigskip

\noindent \textbf{Theorem~\ref{main}}(Second part).  Let $T_n$ be a tree on $n$
vertices. Then 
$$|\End(T_n)|\leq |\End(S_n)|.$$
If $T_n\neq S_n$ then strict inequality holds.
\bigskip

\begin{proof} For the sake of simplicity we prove the statement for $n\geq
  17$. The same proof applies to $n<17$, we only  need to compute a bit more
  carefully. In the end of the proof we will give the details of this more
  precise calculation.

Note that $|\mbox{End}(S_n)|=(n-1)^{n-1}+(n-1)$.

Let $T_n$ be a tree on $n$ vertices and let $d=d_1\geq d_2\geq \dots d_n$ be
its degree sequence. Note that $d_1+d_2\leq n$, since the tree has only $n-1$
edges and the stars corresponding to the first two largest degrees can share
at most one common edge.

First we prove that $|\mbox{End}(T_n)|\leq nd^{n-1}$. To see it, let
$u_1,\dots u_n$ be the vertices of the tree $T_n$ such that $u_1,\dots ,u_k$
induces a tree for every $k$. Then we can chose the image of $u_1$ by $n$
ways, and if we have already chosen the image of $u_1,\dots ,u_{k-1}$, then we
can chose the image of $u_k$ in at most $d$ ways, since it must be the
neighbor of some previous vertex. This means that  $|\mbox{End}(T_n)|\leq
nd^{n-1}$. 

If $d\leq 2n/3$ then
$$nd^{n-1}\leq n\left(\frac{2n}{3}\right)^n\leq (n-1)^{n-1}$$
if $n\geq 17$, since then 
$$\left(\frac{3}{2}\right)^n\geq en^2\geq
n^2\left(1+\frac{1}{n-1}\right)^{n-1}.$$ 

So we can assume that $d\geq \frac{2n}{3}$. Set $d=n-k$. We can assume that
$T_n\neq S_n$, consequently $k\geq 2$.
Let $v_1$ be the vertex having the largest degree and $v_2,\dots ,v_{d+1}$ its
neighbors. Now we can decompose the set of endomorphisms
according to the image of $v_1$ is $v_1$ or not. If it is $v_1$ then there can
be at most $d^{n-1}$ such endomorphisms. If the image of $v_1$ is not $v_1$,
then we can chose that image in at most $(n-1)$ ways and the image of
$v_2,v_3,\dots ,v_{d+1}$ can be chosen at most $d_2$ times and the image of
all other vertices can be chosen in at most $d$ ways. Hence
$$|\mbox{End}(T_n)|\leq d^{n-1}+(n-1)d_2^{d}d^{n-1-d}.$$
All we need to prove is that if $d\leq n-2$ then
$$d^{n-1}+(n-1)d_2^{d}d^{n-1-d}\leq (n-1)^{n-1}+(n-1).$$
With the notations $d=n-k$ we have
$$d^{n-1}+(n-1)d_2^{d}d^{n-1-d}=(n-k)^{n-1}+(n-1)k^{n-k}(n-k)^{k-1}.$$
By the binomial theorem we have
$$(n-1)^{n-1}=(n-k+k-1)^{n-1}\geq (n-k)^{n-1}+(n-1)(n-k)^{n-2}(k-1).$$
It is enough to prove that $(n-k)^{n-2}\geq k^{n-k}(n-k)^{k-1}$. This is
equivalent with $(n-k)^{n-k-1}\geq k^{n-k}$ and it is true since it is
equivalent with
$$\left(\frac{n}{k}-1\right)^{n-k}\geq 2^{n-k}\geq n-k.$$
In the last step we have used that $n/k\geq 3$. 

It is clear from the proof that we only have to check whether one of the
inequalities hold for some $d$:
$$n\leq \left(\frac{n-1}{d}\right)^{n-1}\ \ \mbox{or}\ \ \ 
\left(\frac{n}{k}-1\right)^{n-k}\geq n-k.$$ 
For $8\leq n\leq 16$ it is easy to see that if $d\leq n-4$ then the first
inequality holds and if $d>n-4$, equivalently $k\leq 3$ then the second
inequality holds. For $n=5,6,7$ the first inequality holds if $d\leq n-3$, and
the second inequality holds if $d>n-3$, equivalently $k\leq 2$. For $n=4$ the
claim is trivial $30=|\mbox{End}(S_4)|>|\mbox{End}(P_4)|=16$.
\end{proof}

\subsection{The extremality of path}

\begin{Th} \label{geq 4 leaves}
Let $T_m$ and $T_n$ be trees on $m$ and $n$ vertices, respectively.
If the tree $T_n$ has at least four leaves, then
$$\hom(T_m,T_n)\geq (n-2)2^{m-1}+2.$$
\end{Th}

An easy consequence of this theorem is the following. 

\begin{Cor} \label{4l} If $T_n$ is a tree on $n$ vertices with at least $4$
  leaves, then 
$$\hom(T_m,T_n)\geq \hom(T_m,P_n).$$
\end{Cor}

\begin{proof} Indeed,
$$\hom(T_m,T_n)\geq (n-2)2^{m-1}+2=\hom(S_m,P_n)\geq \hom(T_m,P_n),$$ 
where the second inequality follows from Theorem~\ref{of stars}.
\end{proof}

A consequence of this theorem is that path has the minimal number of
endomorphisms. 
\bigskip

\noindent \textbf{Theorem~\ref{main}}(First part). For all trees $T_n$ on $n$
vertices we have 
$$|End(T_n)|\geq |End(P_n)|.$$

\begin{proof} If $T_n$ has at least four leaves, then
$$\hom(T_n,T_n)\geq \hom(T_n,P_n)\geq \hom(P_n,P_n),$$
where the first inequality follows from Corollary~\ref{4l}, while the second
inequality follows from Theorem~\ref{into paths}.
If the tree $T_n$ has exactly three leaves, then it is star-like. Hence we can
use Theorem~\ref{of starlike} to prove the first inequality:
$$\hom(T_n,T_n)\geq \hom(T_n,P_n)\geq \hom(P_n,P_n).$$
\end{proof}

The proof of Theorem~\ref{geq 4 leaves} will be given next, which would
complete the proof of Theorem~\ref{main}. 
\bigskip

First, we prove a reduction lemma which says that we only have to prove
Theorem~\ref{geq 4 leaves} for trees with exactly $4$ leaves.

\begin{Lemma}[Reduction lemma]\label{reduction} Let $T_m$ be a tree on $m$
vertices and let $n$ be fixed.
Assume that for all tree $T_k$ we have  
$$\hom(T_m,T_k)\geq (k-2)2^{m-1}+2,$$
where $k<n$ and $T_k$ has at least four leaves, or $k=n$ and $T_k$ has
exactly four leaves. Then for any tree $T_n$ on $n$ vertices with at least $4$
leaves we have 
$$\hom(T_m,T_n)\geq (n-2)2^{m-1}+2.$$
\end{Lemma}

In the proof of this lemma we will subsequently use the following very simple
fact. 

\bigskip
\noindent \textbf{Fact.} If $G$ is a graph and $G_1,G_2$ are induced subgraphs
of $G$ with possible intersection, then for any graph $H$ we have
$$\hom(H,G)\geq \hom(H,G_1)+\hom(H,G_2)-\hom(H,G_1\cap G_2).$$
\bigskip

\begin{proof}[Proof of the lemma.] We can assume that $m\geq 2$. Assume that
  $T_n$ is tree with at least $5$ leaves. Otherwise we have nothing to prove.

Let us call a path maximal in $T_n$ if it
connects leaves. If a maximal path contains $k$ vertices of degree at least
$3$, then we say that the maximal path has $k$ branches.

First we prove the statement if $T_n$ contains a maximal path with at least
$3$ branches. Let $v_0Pv_r$ be a maximal path with vertices $u_1,\dots ,u_k$
having degree at least $3$. Let $B_1,\dots ,B_k$ be the branches which we get
if we delete all vertices and edges of the path $v_0Pv_r$ except $u_1,\dots
,u_k$. So $B_i$ is a rooted tree with root $u_i$. Let $u_2^-$ and $u_2^+$ be
the two neighbors of $u_2$ on the path $v_0Pv_r$. Let $T_2$ be the tree
induced by the vertices $V(B_2)\cup \{u_2^-,u_2^+\}$. Let $|V(T_2)|=t$.
We distinguish two cases.

\begin{figure}[h!]
\begin{center}
\scalebox{.65}{\includegraphics{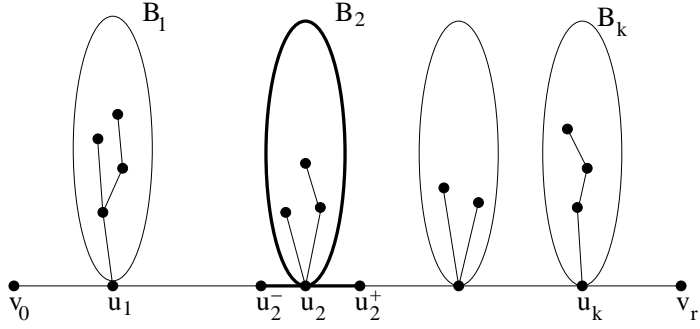}}\caption{A path with
  branches.}    
\end{center}
\end{figure}

\medskip

\noindent \textit{1. case.} In this case we assume that
$\hom(T_m,T_2)<(t-2)2^{m-1}+2$. Consider the  following trees $G_1$ and
$G_2$. $G_1$ is the tree spanned by the vertices $v_0Pu_2^+$ and the branches
$B_1,B_2$. $G_2$ is the tree spanned by the vertices $u_2^-Pv_r$ and the
branches $B_2,\dots, B_k$. Note that $G_1\cup G_2=T_n$, $G_1\cap G_2=T_2$  and
$G_1,G_2$ contains at least $4$ leaves, because $k\geq 3$. By the hypothesis
of the lemma we have
$$\hom(T_m,G_i)\geq (|V(G_i)|-2)2^{m-1}+2$$
for $i=1,2$.  Hence
$$\hom(T_m,T_n)\geq \hom(T_m,G_1)+\hom(T_m,G_2)-\hom(T_m,G_1\cap G_2)\geq $$
$$\geq (|V(G_1)|-2)2^{m-1}+2+(|V(G_2)|-2)2^{m-1}+2-((|V(G_1\cap
G_2)|-2)2^{m-1}+2)=$$ 
$$=((|V(G_1\cup G_2)|-2)2^{m-1}+2=(n-2)2^{m-1}+2.$$
In this case we are done.
\bigskip

  \noindent \textit{2. case.} In this case we assume that
$\hom(T_m,T_2)\geq (t-2)2^{m-1}+2$.  Consider the  following trees $G_1$ and
$G_2$. $G_1$ is the tree spanned by the vertices $(V(T_n)\setminus V(T_2))\cup 
\{u_2^-,u_2,u_2^+\}$. $G_2$ is simply $T_2$.  Note that $G_1\cup G_2=T_n$,
$G_1\cap G_2=\{u_2^-,u_2,u_2^+\}=P_3$ and $G_1$ contains at least $4$
vertices. By the hypothesis
of the lemma we have
$$\hom(T_m,G_1)\geq (|V(G_1)|-2)2^{m-1}+2.$$
We also know that in this case
$$\hom(T_m,G_2)\geq (|V(G_2)|-2)2^{m-1}+2.$$
Note that 
$$\hom(T_m,P_3)\leq \hom(S_m,P_3)=2^{m-1}+2.$$
Then 
$$\hom(T_m,T_n)\geq \hom(T_m,G_1)+\hom(T_m,G_2)-\hom(T_m,G_1\cap G_2)\geq $$
$$\geq (|V(G_1)|-2)2^{m-1}+2+(|V(G_2)|-2)2^{m-1}+2-((|V(G_1\cap
G_2)|-2)2^{m-1}+2)=$$ 
$$=((|V(G_1\cup G_2)|-2)2^{m-1}+2=(n-2)2^{m-1}+2.$$
In this case we are done too.
\bigskip

Hence if one of the maximal path of $T_n$ has at least $3$ branches then we
are done. We still have to consider the trees $T_n$ where all maximal paths
have at most $2$ branches. In the following we show that they have quite simple
structure: they are starlike or double starlike trees, see
Figure~\ref{doublestarlike}. 

Let $v_1$ be  a vertex of $T_n$ of degree at least $3$. Let us decompose $T_n$
to the branches $B_1,B_2,\dots B_k$ at $v_1$. So $v_1$ is a leaf in the
trees  $B_1,B_2,\dots B_k$. We show that all except at most one of 
$B_1,B_2,\dots ,B_k$ are paths. Assume that, for instance, $B_1,B_2$ are not
paths.  Then they contains at least two leaves of $T_n$: $B_1$ contains
$u_1,u_2$, $B_2$ contains $u_3,u_4$. Then the maximal path $u_1Pu_3$ has at
least three branches: one-one inside the branches $B_1$ and $B_2$ and $B_3$ at
the vertex $v$. If all branches are paths, then we are done: $T_n$ is
starlike. If one of them is not path, say $B_1$, then let us consider the
vertex $v_2\in V(B_1)$ having degree at least $3$ which is closest to
$v_1$. Repeating the previous argument to $v_2$ instead of $v_1$, all except
one branches at $v_2$ must be path and we also know that the branch containing
$v_1$ is not path. Hence the tree is double starlike, where the middle path is
$v_1Pv_2$.   
\begin{figure}[h!] 
\begin{center}
\scalebox{.65}{\includegraphics{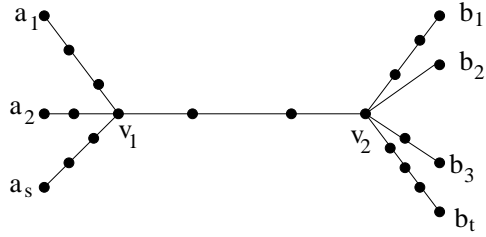}}\caption{A double
  starlike tree.\label{doublestarlike}}    
\end{center}
\end{figure}

We can consider a starlike tree as a double starlike tree, where $v_1=v_2$. 
Let $a_1,\dots ,a_s$ and $b_1,\dots ,b_t$ be the leaves of $T_n$, where 
$a_1,\dots ,a_s$ are closer to $v_1$ than to $v_2$, while $b_1,\dots ,b_t$ are
closer to $v_2$ than to $v_1$. If $v_1=v_2$ we just decompose the set of
leaves into two sets of (almost) equal size. Note that $s,t\geq 2$. Since we
can assume that there are at least $5$ leaves, we assume that $s+t\geq 5$.

If $u_1,\dots ,u_l$ are some vertices of a tree, then we say that the tree
spanned by $u_1,\dots ,u_l$ is the smallest subtree which contains the
vertices $u_1,\dots ,u_l$. It is 
$$span(u_1,\dots ,u_l)=\cup_{1\leq i,j\leq l}u_iPu_j.$$

If $s\geq 3$ and $t\geq 3$ both hold, then let $G_1$ be the tree spanned by
the vertices $a_1,a_2,b_1,b_2$, and let  $G_2$ be the tree spanned by
the vertices $a_2,\dots a_s,b_2,\dots ,b_t$. Then $G_1\cup G_2=T_n$, $G_1\cap
G_2=a_2Pb_2$ and both $G_1,G_2$ have at least $4$ leaves. Since
$$\hom(T_m,G_1\cap G_2)\leq \hom(S_m,G_1\cap G_2)=$$
$$=\hom(S_m,a_2Pb_2)=(|V(G_1\cap G_2)|-2)2^{m-1}+2,$$ 
we have
$$\hom(T_m,T_n)\geq \hom(T_m,G_1)+\hom(T_m,G_2)-\hom(T_m,G_1\cap G_2)\geq $$
$$\geq (|V(G_1)|-2)2^{m-1}+2+(|V(G_2)|-2)2^{m-1}+2-((|V(G_1\cap
G_2)|-2)2^{m-1}+2)=$$ 
$$=(|V(G_1\cup G_2)|-2)2^{m-1}+2=(n-2)2^{m-1}+2.$$
Hence we are done in this case.

If $s\geq 4$ and $t=2$ then let $G_1$ be the tree spanned by $a_1,a_2,b_1,b_2$
and let $G_2$ be the tree spanned by $a_2,\dots ,a_k,b_1$. Then $G_1\cup
G_2=T_n$, $G_1\cap G_2=a_2Pb_1$ and both $G_1,G_2$ have at least $4$ leaves.
In this case we are done as before. Clearly, the case $s\geq 2$ and $t\geq 4$
is completely similar.

The last case is $s=3,t=2$ (and $s=2,t=3$). Let $G_1=span(a_1,a_2,b_1,b_2)$,
$G_2=span(a_2,a_3,b_1,b_2)$, $G_3=span(a_2,b_1,b_2)$,
$G_4=(a_1,a_2,a_3,v_2)$. Then $G_1\cap G_2=G_3$, $G_3\cap G_4=a_2Pv_2$. Note
that $G_1,G_2,G_4$ has $4$ leaves, thus 
$$\hom(T_m,G_i)\geq  (|V(G_i)|-2)2^{m-1}+2$$
for $i=1,2,4$. 
If $\hom(T_m,G_3)\leq  (|V(G_3)|-2)2^{m-1}+2$, then from $T_n=G_1\cup G_2$,
$G_3=G_1\cap G_2$ we obtain that $\hom(T_m,T_n)\geq (n-2)2^{m-1}+2$.
If $\hom(T_m,G_3)\geq  (|V(G_3)|-2)2^{m-1}+2$, then from $T_n=G_3\cap G_4$
we obtain that  $\hom(T_m,T_n)\geq (n-2)2^{m-1}+2$.
Hence we are done in this case as well.
\end{proof}

\begin{Lemma} \label{weak-4leaves} Let $T_n$ be a tree with exactly $4$ leaves
and two vertices of degree $3$. Let $x$ and $y$ be the vertices of $T_n$ with
degree $3$. Assume that there are at most $3$ vertices of $T_n$ which have
degree $2$ and not on the path $xPy$. Then for any tree $T_m$ on $m$ vertices
we have 
$$\hom(T_m,T_n)\geq (n-2)2^{m-1}+2,$$
where $n$ is the number of vertices of $G$.
\end{Lemma}

\begin{proof} We can assume that $m\geq 4$, otherwise the statement is trivial.
We prove the slightly stronger inequality
$$\hom(T_m,T_n)> \left(n-2+\frac{1}{8}\right)2^{m-1}.$$
If $m\geq 4$, then this implies that
$$\hom(T_m,T_n)>(n-2)2^{m-1}+1$$
or equivalently,
$$\hom(T_m,T_n)\geq (n-2)2^{m-1}+2.$$
To prove this statement we use Theorem~\ref{markov} with a suitable Markov
chain.
Let $p_{ij}=\frac{1}{2}$ for $(i,j)\in E(T_n)$ if $i$ has degree $2$. Naturally,
$ p_{ij}=1$ if $(i,j)\in E(G)$ and $i$ is a leaf. Finally, if $i\in \{x,y\}$,
  $j\in xPy$ then $p_{ij}=\frac{1}{2}$ and if $i\in \{x,y\}$, $j\notin xPy$
then $p_{ij}=\frac{1}{4}$. 

Let $r$ be the number of vertices of $xPy$. Then $n=r+t+4$. Let $N=4r+2t+4$.
Then the stationary distribution is the following: $q_i=\frac{4}{N}$ if $i\in
xPy$, $q_i=\frac{2}{N}$ if $i\notin xPy$, but has degree $2$ and finally, $q_i=
\frac{1}{N}$ if $i$ is a leaf.

Then
$$H(P|Q)=\frac{N-12}{N}\log 2+\frac{8}{N}\left(\frac{1}{2}\log 2+\frac{1}{2}
\log 4\right)=\log 2.$$
On the other hand,
$$H(Q)+2(H(D|Q)-H(P|Q))=$$
$$=\left(\frac{4r}{N}\log \frac{N}{4}+\frac{2t}{N}\log
\frac{N}{2}+\frac{4}{N}\log \frac{N}{1}\right)+2\frac{8}{N}\left(\log
3-\frac{3}{2}\log 2\right)=$$
$$=\log \frac{N}{4}+\frac{2t}{N}\log 2+\frac{16(\log 3-\log
  2)}{N}.$$
Note that
$$\log \left(n-2+\frac{1}{8}\right)-\log \frac{N}{4}\leq
\int_{N/4}^{n-2}\frac{1}{x}dx\leq \frac{n-2+\frac{1}{8}-N/4}{N/4}=$$
$$=\frac{4}{N}\left(\frac{t}{2}+1+\frac{1}{8}\right)=
\frac{1}{N}\left(2t+\frac{9}{2}\right).$$  
Hence if 
$$\frac{1}{N}\left(2t+\frac{9}{2}\right)\leq \frac{2t}{N}\log 2+\frac{16(\log
  3-\log 2)}{N},$$
then
$$\log \left(n-2+\frac{1}{8}\right)\leq H(Q)+2(H(D|Q)-H(P|Q)),$$
consequently
$$\hom(T_k,G)\geq \exp(H(Q)+2H(D|Q)+(m-3)H(P|Q))> \left(n-2+\frac{1}{8}\right)2^{m-1}.$$
The above inequality is satisfied if
$$t\leq \frac{8\log \frac{3}{2}-\frac{9}{4}}{1-\log 2}\approx 3.238.$$
This proves the statement of the theorem.
\end{proof}

\begin{Lemma} \label{weak-4leaves2} Let $T_n$ be a tree obtained from a path on
$n-8$ vertices by gluing one-one $P_5$ at the middle vertices to both ends
of the path $P_{n-8}$ (see Fig.~\ref{two}). Then for any tree $T_m$ on $m$
vertices we have 
$$\hom(T_m,T_n)\geq (n-1)2^{m-1}.$$
\end{Lemma}

\begin{proof}  We will show by induction on $m$ that 
\begin{align}\label{three}
\hom(T_m,T_n)\geq (n-1)2^{m-1}.
\end{align}
Let $v$ be any leaf of $T_m$ with unique neighbor $u$
and let $T_{m-1}=T_m-v$ be a rooted tree with root $u$.

\begin{figure}[h!] 
\begin{center}
\scalebox{.65}{\includegraphics{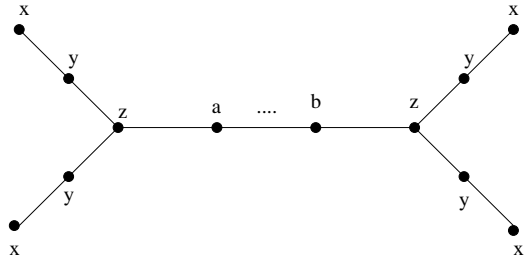}}\caption{Special double
    starlike trees.\label{two}}     
\end{center}
\end{figure}

Let us use the hom-vectors of the Fig.~\ref{two}, that is
\begin{align*}
 {\bf h}(T_{m-1},u,T_n)= (x,x,y,y,z,a,\ldots,b,z,y,y,x,x).
\end{align*}
Now suppose that
$$\hom(T_{m-1},T_n)\geq (n-1)2^{m-2}.$$
It is easy to see by induction that  $z>2x$ if $T_{m-1}$ has at least two
vertices. 
By tree-walk algorithm, we have 
\begin{align*}
 \hom(T_{m}(v),T_n)&=4x+8y+6z+2(a+\cdots+b)\\
 &\geq 8x+8y+4z+2(a+\cdots+b)\\
 &=2\hom(T_{m-1},G)\\
 &\geq(n-1)2^{m-1},
\end{align*}
which shows~\eqref{three}.
\end{proof}

Next we introduce a transformation which we will call LS-switch (Large-Small
switch). 

\begin{figure}[h!]
\begin{center}
\scalebox{.65}{\includegraphics{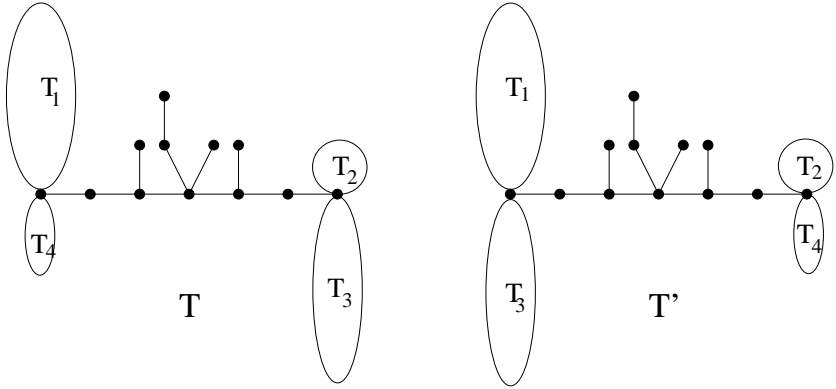}}\caption{LS-switch.}   
\end{center}
\end{figure}

\begin{Def}[LS-switch] Let $R(u,v)$ be a tree with specified vertices $u$ and
$v$ such that the distance of $u$ and $v$ is even and $R$ has an
automorphism of order $2$ which exchanges the vertices $u$ and $v$. Let
$T_1(x),T_2(x),T_3(y),T_4(y)$ be rooted trees such that $T_2(x)$ is the
rooted subtree of $T_1(x)$ and $T_4(y)$ is the rooted subtree of $T_3(y)$. 
Let the tree $T$ be obtained from the trees
$R(u,v),T_1(x),T_2(x),T_3(y),T_4(y)$  by attaching a copy of $T_1(x),T_4(y)$ to
$R(u,v)$ at vertex $u$ and a copy of $T_2(x),T_3(y)$ at vertex $v$.  Assume
that the tree $T'$ is obtained from the trees
$R(u,v),T_1(x),T_2(x),T_3(y),T_4(y)$ by attaching a copy of $T_1(x),T_3(y)$ to
$R(u,v)$ at vertex $u$ and a copy of $T_2(x),T_4(y)$ at vertex $v$. Then $T'$
is the LS-switch of $T$.   Observe that there is a natural bijection between
the color classes of $T'$ and $T$.
\end{Def}

A particular case of the LS-switch is when $R(u,v)$ is a path of even length
with end  vertices $u$ and $v$, $T_2(x)$ and $T_4(y)$ are one-vertex 
rooted trees, then $T'$ is obtained from $T$ by an even-KC-transformation, 
i.e., KC-transformation according to a path of even length. Another useful 
special case is when $R(u,v)$ is a tree where we attach an arbitrary tree to the
middle vertex of the path on $3$ vertices and $u$ and $v$ are the end vertices
of the path (in this case the automorphism simply switches $u$ and $v$), and  
$T_2(x),T_4(y)$ are the rooted trees with $1$ vertex, in this case we
get back to a particular case of the original {\em
Kelmans-transformation}~\cite{kk}.

Actually, the following theorem with respect to the
LS-switch is also true and its proof is just a trivial extension of the even
case of Theorem~\ref{of starlike}.

\begin{Th} Let $T'$ be the LS-switch of $T$. Let $H$ be an arbitrary
tree. Then
$$\hom(H,T)\leq \hom(H,T').$$
\end{Th}

\begin{proof}(Sketch.) The unique shortest path connecting all  $T_i$'s in $T$
(or $T'$) will be denoted by $P_{2k}$, a path of even length with vertices
labeled consecutively by $0,1,\ldots,2k$. Without loss of generality, we can
assume that $0\in V(T_1)$. For $1\leq j\leq 2k-1$, let $A_j$ denote the
component of $T$ that contains the vertex $j$ when we delete all the edges of
$P_{2k}$. By the definition of LS-switch, the subtrees $A_j$ and $A_{2k-j}$
are isomorphic, so we can identify $V(A_j)\setminus\{j\}$ with
$V(A_{2k-j})\setminus\{2k-j\}$. We will also consider $V(T_2)\setminus\{0,2k\}$
as the subset of $V(T_1)\setminus\{0,2k\}$ and $V(T_4)\setminus\{0,2k\}$ as
the subset of $V(T_3)\setminus\{0,2k\}$. 

Let $v$ be a  vertex of $H$. 
For $0\leq s\leq 2k$, $u\in V(T_i)\setminus\{0,2k\}$ ($1\leq i\leq4$) and
$a\in V(A_j)\setminus\{j\}$ ($1\leq j\leq 2k-1$), we define 
$$ p_s:=|\{\{f\in\Hom(H,T) : f(v)=m\}|,\,\, p'_s:=|\{\{f\in\Hom(H,T') :
f(v)=m\}|, $$
$$t_i(u):=|\{f\in\Hom(H,T) : f(v)=u\}|,\,\, t'_i(u):=|\{f\in\Hom(H,T') :
f(v)=u\}|, $$ 
and 
$$p_j(a):=|\{f\in\Hom(H,T) : f(v)=a\}|,\,\, p'_j(a):=|\{f\in\Hom(H,T') :
f(v)=a\}|. $$
We prove by induction that the following inequalities are preserved by the
steps of the tree-walk algorithm. For any
$0\leq s\leq k$, $a\in V(A_j)\setminus\{j\}$ ($1\leq j\leq k$), $u\in
V(T_2)\setminus\{0,2k\}$, $w\in V(T_4)\setminus\{0,2k\}$, $x\in
V(T_1)\setminus V(T_2)$ and $y\in V(T_3)\setminus V(T_4)$ we have
\begin{align}
p'_{k-s}+p'_{k+s}&\geq p_{k-s}+p_{k+s}\,\,\text{and}\,\,p'_{k-s}\geq p_{k+s},
p_{k-s}\label{ine0'}\\ 
p'_{j}(a)+p'_{2k-j}(a)&\geq
p_{j}(a)+p_{2k-j}(a)\,\,\text{and}\,\,p'_{j}(a)\geq p_{2k-j}(a),
p_{j}(a)\label{inequ1'}\\ 
t'_1(u)+t'_2(u)&\geq t_1(u)+t_2(u)\,\,\text{and}\,\,t'_1(u)\geq
t_1(u),t_2(u)\label{inequ2'}\\ 
t'_3(w)+t'_4(w)&\geq t_3(w)+t_4(w)\,\,\text{and}\,\,t'_3(w)\geq
t_3(w),t_4(w)\label{inequ3'}\\ 
t'_1(x)&\geq t_1(x)\,\,\text{and}\,\,t'_3(y)\geq t_3(y)\label{inequ4'}.
\end{align}
We only need to check that the two operations in the tree-walk algorithm
preserve all the above inequalities, which is routine and left to the reader. 
\end{proof}

\begin{Lemma} \label{4-leaves} Let $T_n$ be a tree on $n$ vertices with exactly
four leaves.  Then for any tree $T_m$ on $m$ vertices we have 
$$\hom(T_m,T_n)\geq (n-2)2^{m-1}+2,$$
where $n$ is the number of vertices of $T_n$.
\end{Lemma}

\begin{proof} If $T_n$ has a vertex of degree $4$ then by
  Theorem~\ref{degree-markov} we have
$$\hom(T_m,T_n)\geq 2(n-1)C^{m-2},$$
where
$$C=\left(\prod_{i=1}^nd_i^{d_i}\right)^{1/2e}=2.$$
Hence
$$\hom(T_m,T_n)\geq (n-1)2^{m-1}\geq (n-2)2^{m-1}+2.$$
Now assume that $T_n$ has two vertices, $x$ and $y$, of degree $3$. Among these
trees ($n$ vertices, $4$ leaves, two vertices of degree $3$) let us choose
$\overline{T_n}$ to be the one for which $\hom(T_m,\overline{T_n})$ is minimal
and among these trees the length of the path is maximal. 

Let the four leaves of $\overline{T_n}$ denoted by $z_1,z_2,z_3,z_4$ such that
$z_1,z_2$ are closer to $x$ than $y$, and $z_3,z_4$ are closer to $y$ than
$x$. Let the number of edges of $xPy,xPz_1,xPz_2,yPz_3,yPz_4$ be $a,b,c,d,e$,
respectively. We show that $\max(b,c,d,e)\leq 2$. Indeed, if say $b>2$ then
$\overline{T_n}$ can be obtained by an LS-switch from a graph $T_n^*$ as follows.  

If $b$ is even, then let $u$ be the unique vertex such that $d(z_1,u)=2$. Then
$uPx=R(u,x)$ is a path of even length. Let $T_2=z_1Pu$. Furthermore, let $T_1$
be the tree spanned by the vertices $x,z_3,z_4$. $T_3=xPz_2$ and
$T_4=\{u\}$. Then $T_2$ is a rooted subtree of $T_1$ and $T_4$ is a rooted
subtree of $T_3$. Now making an inverse LS-swith we obtain $T_n^*$.
We know that 
$$\hom(T_m,\overline{T_n})\geq \hom(T_m,T_n^*)$$
and in $T_n^*$, the vertices of degree $3$, $y$ and $u$, has distance $a+b-2>a$
contradicting the choice of $\overline{T_n}$. 

If $b$ is odd, then  let $u$ be the unique neighbor of $z_1$ and we repeat the
previous argument.  The distance of $u$ and $x$ is even again.

Hence we can assume that $\max(b,c,d,e)\leq 2$. If not all of them are $2$,
then we can use Lemma~\ref{weak-4leaves} to get that
$$\hom(T_m,\overline{T_n})\geq (n-2)2^{m-1}+2.$$
If $b=c=d=e=2$, then we use Lemma~\ref{weak-4leaves2} to obtain that
$$\hom(T_m,\overline{T_n})\geq (n-1)2^{m-1}\geq (n-2)2^{m-1}+2.$$
This proves the statement.
\end{proof}

\begin{proof}[Proof of Theorem~\ref{geq 4 leaves}] The statement immediately
  follows from Lemma~\ref{4-leaves} and Lemma~\ref{reduction}.
\end{proof}

\subsection{Trees with 3 leaves}

\begin{Lemma} \label{3-leaves} (a) Let $n=a+b+c+1$, and $\min(a,b,c)\geq
  2$. Then for any tree $T_m$ on $m$ vertices we have
$$\hom(T_m,Y_{a,b,c})\geq (n-2)2^{m-1}+2.$$
\medskip

(b) Let $n=a+b+2$, and $\min(a,b)\geq 3$.  Then for any tree $T_m$ on $m$
vertices we have 
$$\hom(T_m,Y_{a,b,1})\geq (n-2)2^{m-1}+2.$$
\end{Lemma}

\begin{proof} 
\noindent (a) 
\begin{figure}[h!]
\begin{center}
\scalebox{.65}{\includegraphics{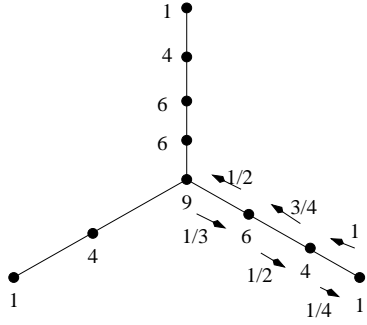}}\caption{$Y_{a,b,c}$
  where $\min(a,b,c)\geq 2$ with a special Markov chain.}     
\end{center}
\end{figure}

We can think to $Y_{a,b,c}$ with $\min(a,b,c)\geq 2$ as follows: we consider
$Y_{2,2,2}$ and we subdivide the edges between the vertex of degree $3$ and
its neighbors a few times. Let us write the weights $1,4,1,4,1,4,9$ to the
vertices of $Y_{2,2,2}$ according to the figure and then let us write weights
$6$ on the new vertices obtained by subdivision. It is easy to check that
there is a unique Markov chain on $Y_{a,b,c}$, where the stationary
distribution is proportional to the weights. (In fact, we write a few
transition probabilities on the figure.)

It is easy to check that $H(P|Q)=\log 2$ and if $N=24+6(n-7)=6(n-3)$, then
$$H(Q)+2(H(D|Q)-H(P|Q))=$$
$$=\left(\frac{9}{N}\log \frac{N}{9}+\frac{12}{N}\log
\frac{N}{4}+\frac{3}{N}\log \frac{N}{1}+\frac{6(n-7)}{N}\log
\frac{N}{6}\right)+$$
$$+2\cdot\frac{12}{N}\left(\log 2-\left(\frac{1}{4}\log 4+\frac{3}{4}\log \frac{4}{3}\right)\right)=$$
$$=\log \frac{N}{6}+\frac{9}{N}\log \frac{6}{9}+\frac{12}{N}\log
\frac{6}{4}+\frac{3}{N}\log \frac{6}{1}+\frac{24}{N}\left(\log
2-\left(\frac{1}{4}\log 4+\frac{3}{4}\log \frac{4}{3}\right)\right)=$$
$$=\log (n-3)+\frac{24}{N}\log \frac{3}{2}.$$
Since 
$$\log (n-2+\varepsilon)-\log
(n-3)=\int_{n-3}^{n-2+\varepsilon}\frac{dx}{x}\leq
\frac{1+\varepsilon}{n-3}=\frac{6}{N}(1+\varepsilon)$$
we can choose $\varepsilon=4\log \frac{3}{2}-1>\frac{1}{2}$ to deduce that
$$\hom(T_m,Y_{a,b,c})\geq (n-2+\varepsilon)2^{m-1}.$$
This is already greater than $(n-2)2^{m-1}+2$ for $m\geq 3$. The statement is
trivial for $m\leq 2$.
\bigskip

\noindent (b)

\begin{figure}[h!]
\begin{center}
\scalebox{.65}{\includegraphics{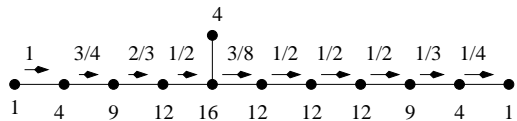}}\caption{$Y_{a,b,1}$
  where $\min(a,b)\geq 3$ with a special Markov chain.}     
\end{center}
\end{figure}

We use completely the same argument as in part (a). We think to $Y_{a,b,1}$ as
a subdivision of $Y_{3,3,1}$ and use the Markov chain on the figure. Again we
have $H(P|Q)=\log 2$ and the sum of the weights is
$N=48+12(n-8)=12(n-4)$. Hence  
$$H(Q)+2(H(D|Q)-H(P|Q))=\log \frac{N}{12}+\frac{12}{N}\left(\frac{11}{3}\log
3-\frac{1}{3}\log 2\right).$$

Since 
$$\log (n-2+\varepsilon)-\log
(n-4)=\int_{n-4}^{n-2+\varepsilon}\frac{dx}{x}\leq
\frac{2+\varepsilon}{n-4}=\frac{12}{N}(2+\varepsilon)$$
we can choose $\varepsilon=\frac{11}{3}\log 3-\frac{1}{3}\log 2-2>1.79$ to
deduce that 
$$\hom(T_m,Y_{a,b,c})\geq (n-2+\varepsilon)2^{m-1}.$$
This is already greater than $(n-2)2^{m-1}+2$ for $m\geq 2$. The statement is
trivial for $m=1$.

\end{proof}

\begin{Rem} Since every Markov chains are reversible on a tree, there is a
natural way to define a new Markov chain on a subdivided edge. Assume that the
probabilities of the stationary distribution were $q_i,q_j$ and
$p_{ij},p_{ji}$ were the transition probabilities at the vertices $i,j$. Then
$q_ip_{ij}=q_jp_{ji}$ (reversibility) and we can put a vertex $r$ with
weight $2q_ip_{ij}$ and $p_{ri}=p_{rj}=1/2$ on the edge $(i,j)$. Then the
new stationary distribution will be proportional to the weights 
$\{q_i\ | i\in V(T)\}\cup \{2q_ip_{ij}\}$.
\end{Rem} 
\bigskip

\noindent \textbf{Theorem~\ref{minimality-path}}. Let $T_n$ be a tree on $n$
vertices. Assume that for a tree $T_m$ we have
$$\hom(T_m,T_n)<\hom(T_m,P_n).$$
Then $T_n=Y_{1,1,n-3}$ and $n$ is even.
\bigskip

\begin{proof} Note that if $T_n$ has at least $4$ leaves then 
  Theorem~\ref{4-leaves} implies that
$$\hom(T_m,T_n)\geq (n-2)2^{m-1}+2=\hom(S_m,P_n)\geq \hom(T_m,P_n)$$
contradicting to the condition of the theorem.
Hence $T_n=Y_{a,b,c}$ for some $a,b,c$. Observe that if one of $a,b,c$ is even
then $Y_{a,b,c}$ can be obtained from $P_n$ by an even-KC-transformation and
then Theorem~\ref{of starlike} implies
$$\hom(T_m,T_n)\geq \hom(T_m,P_n)$$
contradicting to the condition of the theorem. Note that if $n$ is odd, then
one of $a,b,c$ is necessarily even and so we are done. From
Lemma~\ref{3-leaves} we also know that $\min(a,b,c)=1$, say $c=1$ and
$\min(a,b)\leq 2$. But then $\min(a,b)=1$, because it must be even. Hence
$T_n=Y_{1,1,n-3}$ and $n$ is even.
\end{proof}

\begin{Rem} There is a tree $T_m$ for which $\hom(T_m,S_4)<\hom(T_m,P_4)$. 

\begin{figure}[h!]
\begin{center}
\scalebox{.65}{\includegraphics{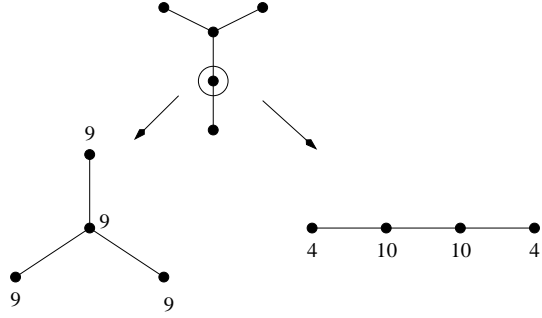}}\caption{An example.}     
\end{center}
\end{figure}

On the figure one can see a rooted tree and its homomorphism vectors to $S_4$
and $P_4$. Now if we attach $k$ copies of this rooted tree at the root then
for the obtained tree $T_m$ we have
$$\hom(T_m,P_4)=2\cdot 4^k+2\cdot 10^k>4\cdot 9^k=\hom(T_m,S_4)$$
for large enough $k$.
\medskip

On the other hand, it seems that $\hom(T_m,Y_{1,1,n-3})\geq \hom(T_m,P_n)$ if
$n\geq 6$ and even. 
\end{Rem}

\section{Homomorphisms into a path}
\label{homtopath}

In this section we study the extremal problem about the number of
homomorphisms from trees to a fixed path. The main purpose of this section is
to prove Theorem~\ref{into paths}. As the theorem suggests we have to
distinguish whether $n$ is even or odd.

\subsection{The KC-transformation: the case of even $n$}

In this section we will prove the part (i) of Theorem~\ref{into paths}.
 
 First we need some new definitions and lemmas.

\begin{Def} A vector ${\bf a}=(a_{1}, a_{2}, \dots, a_{n})$ is
\emph{symmetric} if $a_{i}=a_{n-i+1}$ for $1\leq i \leq n-1$, and
\emph{unimodal} if $a_{1}\leq a_{2}\leq \dots \leq a_{j}\geq a_{j+1}\geq
\dots \geq a_{n}$ for some $j$. Denote by $R_n$ the set of symmetric
positive integer vectors of dimension $n$. For any ${\bf a}, {\bf b}\in
 R_{n}$, define the \emph{dominance   order} on $R_{n}$ by 
 $$
 {\bf a} \preceq{\bf b}\Leftrightarrow \sum_{i=k}^{n+1-k}a_{i} \leq \sum_{i=k}^{n+1-k}b_{i}
\quad\text{for}\quad 1\leq k \leq \lceil\frac{n}{2}\rceil.
$$ 
It is clear that $R_{n}$ is a poset with respect to this order and ${\bf a}
\preceq {\bf b}$ implies $\lVert{\bf a}\rVert\leq \lVert{\bf b}\rVert$.  
Let $U_{n}$ be the set of all unimodal vectors in $R_n$. 
\end{Def} 

\begin{Lemma} \label{le: 2}
Let ${\bf c}\in U_n$ and ${\bf a}, {\bf b}\in R_{n}$ such that ${\bf a}\preceq
{\bf b}$.  Then ${\bf a}\ast{\bf c}\preceq {\bf b}\ast{\bf c}$. 
\end{Lemma}

\begin{proof}
It is clear that both ${\bf a}\ast{\bf c}$ and ${\bf b}\ast{\bf c}$ are in
$R_{n}$. For $1\leq k\leq \lceil\frac{n}{2}\rceil$, by the symmetry  property
of ${\bf a,b}$ and ${\bf c}$, we have the identity  
$$ \sum_{i=k}^{n+1-k}(b_{i}-a_{i})c_i
=c_{k-1}\sum_{i=k}^{n+1-k}(b_{i}-a_{i})+\sum_{j=k}^{\lceil\frac{n}{2}\rceil}(c_j-c_{j-1})\sum_{i=j}^{n+1-j}(b_{i}-a_{i}) $$
where $c_0=0$.  From ${\bf c}\in U_n$ and ${\bf a}\preceq {\bf b}$, we see
that $c_j-c_{j-1}\geq0$ and $\sum_{i=j}^{n+1-j}(b_{i}-a_{i})\geq0$  for $1\leq
j\leq \lceil\frac{n}{2}\rceil$. Thus the right-hand side of the above identity
is nonnegative, which is equivalent to  ${\bf a}\ast{\bf c}\preceq {\bf
  b}\ast{\bf c}$. 
\end{proof}

It is clear that if $a \in U_{2n}$ then $aA^k\in U_{2n}$ for any positive
integer $k$, where $A$ is the adjacency matrix of $P_{2n}$. But for the path
with odd vertices, this is not true in general.

\begin{Lemma} \label{le: 6}
 Let $A$ be the adjacency matrix of $P_{n}$ and $l$ a positive integer. If
 ${\bf a}\in U_n$, then ${\bf a}A^{l}\preceq {\bf a}\ast ({\bf 1_{n}}A^{l})$. 
\end{Lemma}

\begin{proof} 
 Let ${\bf a}=(a_{1}, a_{2}, \dots, a_{n})$, ${\bf a}A^{l}=(x_{1}, x_{2},
 \dots, x_{n})$ and ${\bf a}\ast ({\bf 1_{n}}A^{l})=(y_{1}, y_{2}, \dots,
 y_{n})$. For $1\leq j \leq n$, the sum of coefficients of all $a_{i}$ ($1\leq
 i\leq n$) in $x_j$ is the sum of $j$-th column of $A^{l}$ and the coefficient
 of $a_{j}$ in $y_j$ also equals the sum of $j$-th column of $A^{l}$. It
 follows that the sum of coefficients of all $a_i$ ($1\leq i\leq n$) in
 $\sum_{i=k}^{n+1-k}x_{i}$ equals the sum of coefficients of all $a_i$ ($1\leq
 i\leq n$) in $\sum_{i=k}^{n+1-k}y_{i}$. But for every $i$, the coefficient of
 $a_i$ in $\lVert{\bf a}A^{l}\rVert$ is the sum of $i$-th row of $A^l$ and the
 coefficient of $a_i$ in $\lVert{\bf a}\ast {\bf 1_{n}}A^{l}\rVert$ is the sum
 of $i$-th column of $A^l$, which are equal. Thus $\sum_{i=k}^{n+1-k}x_{i}\leq
 \sum_{i=k}^{n+1-k}y_{i}$ follows from the unimodality of ${\bf a}$, which
 shows that ${\bf a}A^{l}\preceq {\bf a}\ast ({\bf 1_{n}}A^{l})$. 
\end{proof}

Before we prove part (i) of Theorem~\ref{into paths}, we show that we only
have to prove the statement for the case of even $n$ since the following lemma
implies it for the case of odd $n$ if $\diam(T)\leq n-1$.

\begin{Lemma} \label{average}
If $\diam(T)\leq n-1$ then 
\begin{equation}\label{pep0}
\hom(T,P_n)=\frac{1}{2}(\hom(T,P_{n-1})+\hom(T,P_{n+1})).
\end{equation}
\end{Lemma}

\begin{proof}
Let the vertices of $P_n$ be labeled consecutively by $1,2,\ldots,n$. We can
decompose the set $\Hom(T,P_n)$ to the following two sets. The first set
consists of those homomorphisms which do not contain the vertex $n$ in their
image and the second set consists of those homomorphisms which contain the
vertex $n$ in their image. Clearly, the cardinality of the first set is
$\hom(T,P_{n-1})$. The cardinality of the second set will be denoted by
$\hom(T,P_n,n\in f(T))$. So  
\begin{equation}\label{pep1}
\hom(T,P_n)=\hom(T,P_{n-1})+\hom(T,P_n,n\in f(T)).
\end{equation}
We can repeat this argument with the path $P_{n+1}$ as well:
\begin{equation*}
\hom(T,P_{n+1})=\hom(T,P_n)+\hom(T,P_{n+1},n+1\in f(T)).
\end{equation*}
Rearranging this we obtain that 
\begin{equation}\label{pep3}
\hom(T,P_{n})=\hom(T,P_{n+1})-\hom(T,P_{n+1},n+1\in f(T)).
\end{equation}
The crucial observation is that 
$$
\hom(T,P_n,n\in f(T))=\hom(T,P_{n+1},n+1\in f(T)).
$$
Indeed, if $n+1\in f(T)$ then $1\notin f(T)$ because
$\diam(T)<\diam(P_{n+1})$, so all these homomorphisms go to the path
$\{2,3,\ldots,n+1\}$ and therefore there is a natural correspondence between
the two sets. Hence if $\diam(T)\leq\diam(P_n)$ then by adding together
Eq.~\eqref{pep1} and Eq.~\eqref{pep3}  we obtain~Eq.~\eqref{pep0}. 
\end{proof}

\begin{proof}[Proof of part (i) of Theorem~\ref{into paths}] Let $A$ be
  the adjacency matrix of $P_n$. Let $T'_m=KC(T,x,y)$ be the KC-transformation
  of  the tree $T_m$ with respect to a $x$--$y$ path $P$. Let $B_1$ and $B_2$ be
  the components of $y$ and $x$ in the subgraph of $T_m$ by deleting all the
  edges of $P$ respectively.  

The case $n$ is even. It is easy to see that an element of $U_n$ multiplied by
$A$ is still in $U_n$ and the Hadamard product of two elements in $U_n$ is
also in $U_n$. By the tree-walk algorithm, any hom-vector from a tree to $P_n$
is in $U_n$. Again by the tree-walk algorithm, we have  
$$
{\bf h}(T_m,x,P_n)=({\bf h}(B_1,y,P_n)A^k)\ast{\bf h}(B_2,x,P_n)
$$
and 
$${\bf h}(T'_m,x,P_n)={\bf h}(B_1,y,P_n)\ast {\bf 1_{n}}A^k\ast{\bf
  h}(B_2,x,P_n),$$ 
where $k$ is the length of the path $P$.
By lemma~\ref{le: 6}, we have 
$$
{\bf h}(B_1,y,P_n)A^k\preceq {\bf h}(B_1,y,P_n)\ast {\bf 1_{n}}A^k.
$$
It follows from lemma~\ref{le: 2} that ${\bf h}(T,x,P_n) \preceq {\bf
  h}(T',x,P_n)$, which implies $\hom(T,P_n)\leq \hom(T',P_n)$.  
 
For the case of odd $n$ we have already seen that Lemma~\ref{average} implies
the statement.
\end{proof}

The KC-transformation does not always increase the number of
  homomorphisms to the path $P_n$ when $n$ is odd. For example, in
  Fig.~\ref{counter2}, we have  
$\hom(T_6,P_3)=20>16=\hom(T'_6,P_3)$.

\begin{figure}[h!]
\begin{center}
\scalebox{.65}{\includegraphics{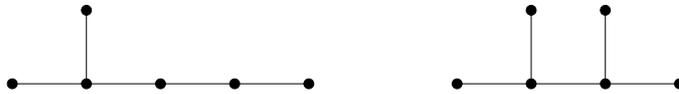}}   
\end{center}
\caption{The trees $T_6$ (left) and $T'_6$ (right).  \label{counter2}}
\end{figure}

\subsection{More tree transformations: the case of odd $n$}
Clearly, if $n$ is even or $n$ is odd and $\diam(T_m)\leq n-1$, then part (i)
of Theorem~\ref{into paths} implies the part (ii). In this subsection, we will
develop some more transformations on  trees to prove part (ii) of
Theorem~\ref{into paths} for odd $n$.   

\textbf{From now on in this subsection, $n$ is  odd.} 

For any $u\in V(T)$, denote by $T(u)$ the rooted tree with a root at $u$. As
usual, we will denote by  ${\bf h}(T,u,G)$ the hom-vector of the rooted tree
$T(u)$ into $G$. One can easily check that the hom-vector of a rooted 
tree into $P_n$, i.e., ${\bf h}(T,u,P_n)$ is not unimodal anymore. On the
other hand, the situation is not as bad as one may think after the first sight.

\begin{Def} We say that $(a_1,a_2,\dots ,a_n)$ is \emph{symmetric bi-unimodal}
if the sequence itself is symmetric and the two subsequences
$(a_1,a_3,a_5,\dots a_n)$ and $(a_2,a_4,\dots a_{n-1})$ are unimodal. 
\end{Def}

\begin{Pro} Let $T(u)$ be a rooted tree and $(a_1,\dots ,a_n)$ be the
hom-vector of $T(u)$ into $P_n$. Then $(a_1,\dots ,a_n)$ is symmetric
bi-unimodal. 
\end{Pro}

\begin{proof} This statement follows from the tree-walk algorithm and the
observation that if $(a_1,a_2,\dots ,a_n)$ and $(b_1,b_2,\dots ,b_n)$ are
symmetric bi-unimodal, then  the sequences  $(a_1b_1,a_2b_2,\dots ,a_nb_n)$
and $(a_2,a_1+a_3,\dots a_{n-2}+a_n,a_{n-1})$ are also symmetric bi-unimodal.
\end{proof} 

The following one is a very surprising theorem. In fact, it is not true for
even $n$.

\begin{Lemma}[Correlation inequality] Let $T_1(u)$ be a rooted tree and $T_2(u)$
be its rooted subtree, i.e., $T_2$ is subtree of $T_1$ and their roots are
the same. Let $(a_1,\dots ,a_n)$ be the hom-vector of $T_1(u)$ into $P_n$
and let $(b_1,\dots ,b_n)$ be the hom-vector of $T_2(u)$ into $P_n$.
Assume that $i\equiv j\ \ (mod \ 2)$ and $|\frac{n+1}{2}-i|\leq
 |\frac{n+1}{2}-j|$, so $i$ is closer to the center than $j$. Then
$$a_ib_j\geq a_jb_i.$$
\end{Lemma}
 
\begin{proof} We need to prove that $\frac{a_i}{a_j}\geq \frac{b_i}{b_j}$.
(From this form it is clear that it is a natural associative ordering on
rooted trees.)  We prove this claim by induction on the number of vertices
of $T_1(u)$. If $|T_1(u)|\leq 2$ then it is trivial to check the statement.

If $u$ is not a leaf in $T_1$ then by induction this inequality holds for the
branches at $u$ and then it is true for the Hadamard products. If $u$ is a
leaf in $T_1$ then so in $T_2$. Let $v$ be the unique neighbor of $u$ and let us
consider the rooted trees $(T_1-u)(v)$ and $(T_2-u)(v)$. Let $(a'_1,a'_2,\dots
,a'_n)$ and $(b'_1,\dots ,b'_n)$ be their hom-vectors. Then,
$$a_i=a'_{i-1}+a'_{i+1}\ \ \text{and}\ \ b_i=b'_{i-1}+b'_{i+1},$$
where $a'_0=a'_{n+1}=b'_0=b'_{n+1}=0$. Note that the numbers $i-1,i+1,j-1,j+1$
are still congruent modulo $2$.
Because of the symmetry we can assume
that $j<i\leq \frac{n+1}{2}$. If in addition $i<\frac{n+1}{2}$ then we have 
$j-1<j+1\leq i-1<i+1 \leq \frac{n+1}{2}$ and we can apply the induction
hypothesis: 
$$a'_{i\pm 1}b'_{j\pm 1}\geq a'_{j\pm 1}b'_{i\pm 1}$$
in all four cases, thus
$$a_ib_j=(a'_{i-1}+a'_{i+1})(b'_{j-1}+b'_{j+1})\geq
(a'_{j-1}+a'_{j+1})(b'_{i-1}+b'_{i+1})=a_jb_i.$$
If $i=\frac{n+1}{2}$ then the above four inequalities are still true, because
$\{i-1,i+1\}$ are still closer to $\frac{n+1}{2}$ than the numbers
$\{j-1,j+1\}$.  
Hence we have proved the statement. 
\end{proof}

\begin{Lemma}[Log-concavity of the hom-vector.] Let $T_1(u)$ be a rooted tree
  and let $(a_1,\dots ,a_n)$ be its hom-vector. Assume that $i<j$ and
  $i\neq j\ (\bmod \ 2)$.
Then $a_ia_j\leq a_{i+1}a_{j-1}$.
\end{Lemma}

\begin{proof} The proof is almost identical to the proof of the Correlation
 inequality and thus is omitted.
\end{proof}

\begin{Def} Let $T_1,T_2$ be trees. Let $u$ be an arbitrary vertex of $T_1$ and
let $A$ and $B$ be the color classes of $T_2$ considered as a bipartite graph. 
Let $h_A$ and $h_B$ be the number of homomorphisms from $T_1$ to $T_2$ where $u$
goes to $A$ and $B$, respectively. Then let
$$g(T_1,T_2):=h_Ah_B.$$
Note that $g(T_1,T_2)$ is independent of the vertex $u$. 
\medskip 

Let the vertices of $P_n$ be labeled consecutively by $1,2,\ldots,n$. If
$T_2=P_n$ then $\hom_0(T_1(u),P_n)$ denotes the number of homomorphisms of 
$T_1$ into $P_n$, where the image of $u$ is a vertex of even index and
$\hom_1(T_1(u),P_n)$ denotes the number of homomorphisms of $T_1$ into $P_n$,
where the image of $u$ is a vertex of odd index. Thus in this case
$$g(T_1,P_n)=\hom_0(T_1(u),P_n)\cdot \hom_1(T_1(u),P_n).$$
\end{Def}

First we prove the following curious theorem on the function $g(T,P_n)$.

\begin{Th} \label{g-path} Let $T_m$ be a tree on $m$ vertices. Then
$$g(T_m,P_n)\geq g(P_m,P_n).$$
\end{Th}

We will prove Theorem~\ref{g-path} by using two transformations: the LS-switch
and the so-called short-path shift.  

Recall that the LS-switch is a generalization of the
even-KC-transformation. For the sake of convenience we repeat its definition
below. 

Let $R(u,v)$ be a tree with specified vertices $u$ and
$v$ such that the distance of $u$ and $v$ is even and $R$ has an
automorphism of order $2$ which exchanges the vertices $u$ and $v$. Let
$T_1(x),T_2(x),T_3(y),T_4(y)$ are rooted trees such that $T_2(x)$ is the
rooted subtree of $T_1(x)$ and $T_4(y)$ is the rooted subtree of $T_3(y)$. 
Let the tree $T$ be obtained from the trees
$R(u,v),T_1(x),T_2(x),T_3(y),T_4(y)$  by attaching a copy of $T_1(x),T_4(y)$ to
$R(u,v)$ at vertex $u$ and a copy of $T_2(x),T_3(y)$ at vertex $v$.  Assume
that the tree $T'$ is obtained from the trees
$R(u,v),T_1(x),T_2(x),T_3(y),T_4(y)$ by attaching a copy of $T_1(x),T_3(y)$ to
$R(u,v)$ at vertex $u$ and a copy of $T_2(x),T_4(y)$ at vertex $v$. Then $T'$
is the LS-switch of $T$.   Observe that there is a natural bijection between
the color classes of $T'$ and $T$.

This transformation seems to be quite general, but still the following theorem
is true.

\begin{Th} \label{LS-Th} Let $T$ be a tree and $T'$ be the LS-switch of
$T$. Then 
$$\hom_k(T'(u),P_n)\geq \hom_k(T(u),P_n)$$
for $k=0,1$. In particular,
$$\hom(T',P_n)\geq \hom(T,P_n)
\quad\text{and} \quad
g(T',P_n)\geq g(T,P_n).$$
\end{Th}

\begin{proof}  Let the vertices of $P_n$ be labeled consecutively by
$1,2,\ldots,n$. For a rooted tree $T(r)$ let $h(T,i)$ denote the number of
homomorphisms of $T$ into $P_n$ such that $r$ goes to the vertex $i$. So
$(h(T,1), h(T,2),\dots ,h(T,n))$ is the hom-vector of $T(r)$ into $P_n$. Let
$a_{ij}$ be the number of homomorphisms of $R(u,v)$ into $P_n$ such that $u$
goes to $i$ and $v$ goes to $j$. Note that $a_{ij}=a_{ji}$ because of the
automorphism of order $2$ of $R$ switching the vertices $u$ and $v$. Also note
that $a_{ij}=0$ if $i$ and $j$ are incongruent modulo $2$ since 
the distance of $u$ and $v$ is even. Observe that for $k=0,1$ we have
$$\hom_k(T'(u),P_n)=\sum_{i,j\equiv
  k\ (2)}a_{ij}h(T_1,i)h(T_3,i)h(T_2,j)h(T_4,j)$$  
and 
$$\hom_k(T(u),P_n)=\sum_{i,j\equiv
  k\ (2)}a_{ij}h(T_1,i)h(T_4,i)h(T_2,j)h(T_3,j).$$  
Using $a_{ij}=a_{ji}$ we can rewrite these equations as follows:
$$\hom_k(T'(u),P_n)=\sum_{i\equiv
  k\ (2)}a_{ii}h(T_1,i)h(T_3,i)h(T_2,i)h(T_4,i)+$$   
$$+\sum_{i<j \atop i,j\equiv k\ (2)}a_{ij}(h(T_1,i)h(T_3,i)h(T_2,j)h(T_4,j)+
h(T_1,j)h(T_3,j)h(T_2,i)h(T_4,i)),$$
and  
$$\hom_k(T(u),P_n)=\sum_{i\equiv k\ (2)}a_{ii}h(T_1,i)h(T_3,i)h(T_2,i)h(T_4,i)+$$
$$+\sum_{i<j \atop i,j\equiv k\ (2)}a_{ij}(h(T_1,i)h(T_4,i)h(T_2,j)h(T_3,j)+
h(T_1,i)h(T_4,i)h(T_2,j)h(T_3,j)).$$
Hence
$$\hom_k(T'(u),P_n)-\hom_k(T(u),P_n)=$$
$$\sum_{i<j \atop i,j\equiv k\ (2)}
a_{ij}(h(T_1,i)h(T_2,j)-h(T_1,j)h(T_2,i))(h(T_3,i)h(T_4,j)-h(T_3,j)h(T_4,i)).$$  
By the correlation inequalities, the signs of 
$$h(T_1,i)h(T_2,j)-h(T_1,j)h(T_2,i)$$ 
and
$$h(T_3,i)h(T_4,j)-h(T_3,j)h(T_4,i)$$
only depend on the positions of $i$ and $j$, so they are the same. Hence
$$\hom_k(T'(u),P_n)-\hom_k(T(u),P_n)\geq 0.$$
\end{proof}

The main problem with the LS-switch is that it preserves the sizes of the color
classes of the tree (considered as a bipartite graph). So we need another
transformation which can help to compare trees with different color class sizes.
For this reason we introduce the following transformation which we will call
{\em short path shift}. 

\begin{figure}[h!]
\begin{center}
\scalebox{.65}{\includegraphics{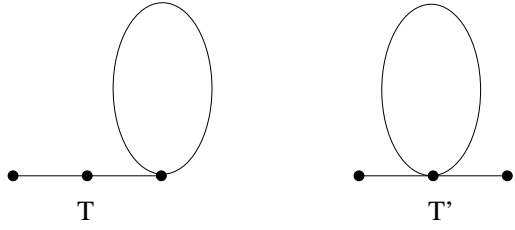}}
\caption{Short-path
  shift.}    
\end{center}
\end{figure}

\begin{Th} Assume that $T'$ is obtained from the rooted tree $T_1(u)$ by
attaching it to the middle vertex of a path on $3$ vertices. Let us assume
that $T$ is obtained from $T_1(u)$ by attaching it to an end vertex of a
path on $3$ vertices.   Then for any odd $n\geq 5$, we have
$$g(T',P_n)\geq g(T,P_n).$$
\end{Th}

We will say that $T'$ is obtained from $T$ by a short-path shift.

\begin{proof}
Let $(a_1,a_2,\dots ,a_n)$ be the hom-vector of $T_1(u)$. Then the hom-vector
of $T(u)$ is 
$$(2a_1,3a_2,4a_3,4a_4,\dots ,4a_{n-2},3a_{n-1},2a_n)$$ 
and the hom-vector of $T'(u)$ is  
$$(a_1,4a_2,4a_3,4a_4,\dots ,4a_{n-2},4a_{n-1},a_n).$$ 
Hence
$$g(T',P_n)=\biggl(4\sum_{i\equiv
  0\ (2)}a_i\biggr)\biggl(-3a_1-3a_n+4\sum_{i\equiv 1\ (2)}a_i\biggr),$$ 
while
$$g(T,P_n)=\biggl(-a_2-a_{n-2}+4\sum_{i\equiv
  0\ (2)}a_i\biggr)\biggl(-2a_1-2a_n+4\sum_{i\equiv 1\ (2)}a_i\biggr).$$
Using the symmetry $a_i=a_{n+1-i}$ we get that
$$g(T',P_n)-g(T,P_n)=8\sum_{i\equiv 1\ (2) \atop 3\leq i\leq n-2}
(a_2a_i-a_1a_{i+1}).$$ 
By the log-concavity of the hom-vector, all $a_2a_i-a_1a_{i+1}\geq 0$ for
$i\equiv 1\ (\bmod 2)$. Hence
$$g(T',P_n)\geq g(T,P_n).$$
\end{proof}

Observe that both transformations decreases the Wiener-index (sum of the
distances of every pair of vertices in $T$) strictly: $W(T')<W(T)$ if $T'$ is
obtained from $T$ by an LS-switch or short-path shift and $T'$ is not
isomorphic to $T$.

\begin{proof}[Proof of Theorem~\ref{g-path}] For $n=3$ we have
$g(T,P_3)=2^{|T|}$, so there is nothing to  prove. Hence we can assume that
  $n\geq 5$. 

Let us consider the tree $T^*_m$ on $m$ vertices for which $\hom(T^*_m,P_n)$ is
minimal and among these trees it has the largest Wiener-index. We show that
$T^*_m$ is $P_m$. Assume for contradiction that $T^*_m$ is not $P_m$. Let $v$
be an end vertex of a longest path of $T^*_m$. Clearly, $v$ is a leaf, let $u$
be its unique neighbor. We distinguish two cases. 

If $\deg(u)\geq 3$ then $T^*_m$ can be decomposed into two branches, one of
which is a star on at least $3$ vertices. (Otherwise, $v$ cannot be an end
vertex of a longest path.) Hence $u$ has another neighbor $w$ which is a leaf.
In this case, $T^*_m$ is an image of a tree $\overline{T}$ by a short-path shift
with respect to the path $vuw$. Hence $g(\overline{T},P_n)\leq g(T^*_m,P_n)$
and  $W(\overline{T})>W(T^*_m)$. This contradicts the choice of $T^*_m$.

If $\deg(u)=2$ then let $w$ be the closest vertex to $v$ having degree at
least $3$. (Such a $w$ must exist, because $T^*_m$ is not $P_m$.) Note that
$d(v,w)\geq 2$. Let us decompose the tree $T^*_m$ into branches $P(v,w)$,
$T_1(w)$ and $T_3(w)$ at the vertex $w$, where  $T_1(w)$ and $T_3(w)$ are
non-trivial trees. If $d(v,w)$ is even, then  $T^*_m$ is an image of  a tree
$\overline{T}$ by an LS-switch, where $T_2(v)=T_4(v)$ are one-vertex trees. If
$d(v,w)$ is odd (consequently $d(u,w)$ is even), then   $T^*_m$ is an
image of  a tree $\overline{T}$ by an LS-switch, where $T_2(u)$ is a one-vertex
tree and $T_4(u)$ is the rooted tree on the vertex set $\{u,v\}$. In both
cases $g(\overline{T},P_n)\leq g(T^*_m,P_n)$ and  $W(\overline{T})\geq
W(T^*_m)$. In the first case, the second inequality is strict contradicting
the choice of $T^*_m$. In the second case, it may occur that $T_3(w)$ is also
an edge implying that $T^*_m=\overline{T}$. By changing the role of $T_2(v)$
and $T_4(v)$, we can ensure that $T_1(w)$ is also an edge (otherwise we get
the same contradiction as before). Hence in this case $T_m^*$ is a path with
an edge attached on the second vertex. In this case we can realize that it is
a short-path shift of a path at the vertex $w$. Hence we get a contradiction
in this case too, which finishes the proof of the theorem.
\end{proof}

To prove part (ii) of Theorem~\ref{into paths}, we will distinguish two cases
according to the parity of $m$. 
\subsubsection{{\bf Trees on even number of vertices}}
 If $m$ is even, then
$g(P_m,P_n)=\frac{1}{4}\hom(P_m,P_n)^2$ since 
$h_A=h_B=\frac{1}{2}\hom(P_m,P_n)$. (However, the color classes of $P_n$ are
not symmetric, but the color classes of $P_m$ are symmetric so it is equally
likely which color class goes to which one.) Hence by Theorem~\ref{g-path} we
have 
$$\hom(T_m,P_n)\geq 2g(T_m,P_n)^{1/2}\geq 2g(P_m,P_n)^{1/2}=\hom(P_m,P_n).$$
If $m$ is odd then we still need to work a bit.

\subsubsection{{\bf Trees on odd number of vertices}} 
From now on $n$ and $m$ are odd. Since $m$ is odd, it makes sense to speak
about large and small color class of the tree considered as a bipartite
graph. If $T_m$ is a tree on $m$ vertices, then $S$ will denote the color
class of size at most $\frac{m-1}{2}$  and $L$ will denote the color
class of size at least $\frac{m+1}{2}$. $S$ and $L$ stands for small and large.
The notation $\hom_0(T(S),P_n)$ denotes   $\hom_0(T(u),P_n)$, where $u\in
S$. Hence it means that the small color class goes to the even-indexed
vertices of the path. We can similarly define $\hom_1(T(S),P_n)$. 

The following is a simple observation.

\begin{Lemma} \label{01-path}
$$\hom_0(P_m(S),P_n)\geq \hom_1(P_m(S),P_n).$$
\end{Lemma}

This lemma asserts that the small class of the path $P_m$ `likes' to go the
large class of the path  $P_n$.

\begin{proof} Let $u$ be a leaf of $P_m$. Note that $u\in L$. Let $v$ be its
  neighbor, and let $(a_1,a_2,\dots ,a_n)$ be the hom-vector of $P_{m-1}(v)$. 
Then the hom-vector of $P_{m}(v)$ is $(a_1,2a_2,\dots ,2a_{n-1},a_n)$. Note that
$$\hom_0(P_m(S),P_n)=\hom_0(P_m(v),P_n)=2\sum_{j\equiv 0\ (2)}a_j,$$
while 
$$\hom_1(P_m(S),P_n)=\hom_1(P_m(v),P_n)=-2a_1+2\sum_{j\equiv 1\ (2)}a_j.$$
Note that 
$$\sum_{j\equiv 0\ (2)}a_j=\sum_{j\equiv 1\ (2)}a_j$$
since $P_{m-1}$ has even number of vertices. Hence
$$\hom_0(P_m(S),P_n)\geq \hom_1(P_m(S),P_n).$$
\end{proof}

The following theorem is the main result of this part of the proof. It will
imply the minimality of the path.

\begin{Th} \label{S-hom} Let $m$ be odd. Then for any tree $T_m$ on $m$
vertices we have 
$$\hom_0(T_m(S),P_n)\geq \hom_0(P_m(S),P_n).$$
\end{Th}
\bigskip

\noindent \textbf{Theorem~\ref{into paths}~(ii) $n,m$ are odd positive integers. } Let $m$ be odd. For any tree $T_m$
  on $m$ vertices we have 
$$\hom(T_m,P_n)\geq \hom(P_m,P_n).$$

\begin{proof} From Theorem~\ref{g-path} we know that
$$g(T_m,P_n)\geq g(P_m,P_n).$$
In other words, 
$$\hom_0(T_m(S),P_n)\hom_1(T_m(S),P_n)\geq
\hom_0(P_m(S),P_n)\hom_1(P_m(S),P_n).$$
By Theorem~\ref{S-hom} and Lemma~\ref{01-path} we have
$$\hom_0(T_m(S),P_n)\geq \hom_0(P_m(S),P_n)\geq \hom_1(P_m(S),P_n).$$
These inequalities together imply that
$$\hom_0(T_m(S),P_n)+\hom_1(T_m(S),P_n)\geq
\hom_0(P_m(S),P_n)+\hom_1(P_m(S),P_n).$$
Hence 
$$\hom(T_m,P_n)\geq \hom(P_m,P_n).$$
\end{proof}

Now we start to prove Theorem~\ref{S-hom}. We need a few lemmas. The first one
is trivial, but crucial.

\begin{Lemma} \label{leaf-move} Let $T_1$ be a tree on $m$ vertices, and let
$u\in L$ be a leaf. Let $v\in S$ an arbitrary vertex of the small class. Let
$T_2$ be a tree obtained from $T_1$ by deleting the vertex $u$ from $T_1$ and
attaching a leaf $u'$ to $v$. (So we simply move a leaf of the large class
to another place, but we take care  not changing the sizes of the color
classes.) Then 
$$\hom_0(T_1(S),P_n)=\hom_0(T_2(S),P_n).$$
\end{Lemma}

\begin{proof} Let $T^*=T_1-u=T_2-u'$. Note that
$$\hom_0(T_1(S),P_n)=2\hom_0(T^*(S),P_n)$$
since any homomorphism of $T^*$ into $P_n$, where the small class goes to
even-indexed vertices, can be extended into a similar homomorphism of $T_1$
in exactly two ways. Similarly,
$$\hom_0(T_2(S),P_n)=2\hom_0(T^*(S),P_n).$$
Hence 
$$\hom_0(T_1(S),P_n)=\hom_0(T_2(S),P_n).$$
\end{proof}

We will use the following transformation too. 

\begin{Def}[Claw-deletion] Let $T'$ be a tree which contains a claw: three
leaves attached to the same vertex. Let $T$ be obtained from $T'$ by deleting
the three leaves and attaching a path of length $3$ to the common neighbor of
the leaves. We call this transformation \textit{claw-deletion}.
\end{Def}

\begin{figure}[h!]
\begin{center}
\scalebox{.65}{\includegraphics{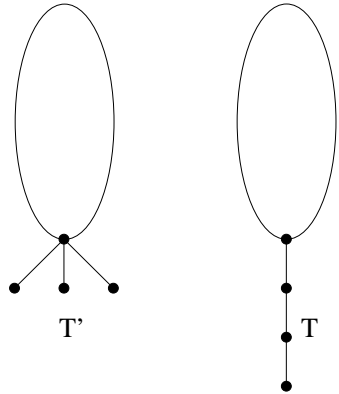}}\caption{Claw-deletion.}    
\end{center}
\end{figure}

Note that the claw-deletion changes the sizes of the color classes. We have
to be careful when we apply it, because it may occur that claw-deletion 
changes the small class into the large one. 

We need the following property of the claw-deletion.

\begin{Lemma} \label{claw-deletion} Let $T'$ be a tree on $m$ vertices with
  color classes $S'$ and $L'$. Assume that $|L'|-|S'|\geq 3$ and $T'$ contains
  a claw. Let $T$ be obtained from $T'$ by a claw-deletion, where we assume
  that the center $v$ of the claw is in the class $S'$. Then
$$\hom_0(T'(S'),P_n)>\hom_0(T(S),P_n).$$
\end{Lemma}

\begin{proof} Note that the condition $|L'|-|S'|\geq 3$ guarantees that the
  small class cannot become the large one in $T$. 
Let $u_1,u_2,u_3$ be the leaves of $T$, and let $v$ be their
common neighbor. Let $T^*=T'-\{u_1,u_2,u_3\}$. Let $(a_1,a_2,\dots ,a_n)$ be the
  hom-vector of $T^*(v)$. Then the hom-vector of $T'(v)$ is 
$$(a_1,8a_2,\dots ,8a_{n-1},a_n),$$
and the hom-vector of $T(v)$ is
$$(3a_1,6a_2,7a_3,8a_4,\dots ,6a_{n-1},3a_n)$$
if $n\geq 7$. If $n=5$ then the hom-vector of $T'(v)$ is 
$(a_1,8a_2,8a_3,8a_4,a_5)$, while the hom-vector of $T(v)$ is
$(3a_1,6a_2,6a_3,6a_4,3a_5)$. In both cases
$$\hom_0(T'(S'),P_n)-\hom_0(T(S),P_n)=4a_2.$$
\end{proof}

Our strategy will be the following. We transform a tree into a path by moving
leaves and using the LS-switch and claw-deletion repeatedly. If we want to
apply this last operation, we need to be sure that the condition
$|L'|-|S'|\geq 3$ holds. The following lemma will be useful.

\begin{Lemma} \label{leaves} Let $T$ be a tree with color classes $A$ and
$B$. Assume that all leaves of $T$ belong to the color class $A$. Then
$|A|\geq |B|$. 
\end{Lemma}

\begin{proof} We prove the statement by induction on the number of
vertices. If $T$ has at most $3$ vertices, the claim is trivial. Assume that
$T$ has $n$ vertices and we have proved the statement for trees on at most
$n-1$ vertices. Let $u_1$ be a leaf of $T$, and let $v$ be its unique
neighbor. Let $u_1,\dots ,u_k$ be the leaves of $T$ adjacent to $v$. Note that
$u_1,\dots ,u_k\in A$ and $v\in B$. Erase $u_1,\dots ,u_k$ from $T$, and let
$T^*$ be the obtained tree. If $v$ is a leaf in the obtained tree then erase
it as well and let us denote the resulted tree by $T^{**}$, otherwise
$T^{**}=T^{*}$.  It may occur that $T^{**}$ is the empty graph, but it
is not problem. For $T^{**}$ it is still true that one color class, $A^{**}$,
contains all leaves. Thus by induction $|A^{**}|\geq |B^{**}|$. Then
$$|A|=|A^{**}|+k\geq |B^{**}|+1\geq |B|,$$
as desired.
\end{proof}

\begin{Rem} \label{leaf-rem} Let $T'$ be a tree on $m$ vertices with color
classes $L$ and $S$. Assume that $T'$ contains a claw, and all leaves belong
to $L$. Then $|L|-|S|\geq 3$. Indeed, if we delete two leaves from the claw
then it is  still true for the resulting tree $T^*$ that all leaves belong
to one color class, and it must be the larger one. Hence $|L^*|\geq |S^*|$,
and since $m$ is odd, we have $|L^*|\geq |S^*|+1$. Thus for the original
tree $T'$ we have $|L|-|S|\geq 3$.
\end{Rem}

\begin{proof}[Proof of Theorem~\ref{S-hom}] 
Let $\mathcal{T}_m$ be the set of trees on $m$ vertices which minimizes
$\hom_0(T(S),P_n)$. From $\mathcal{T}_m$ let us choose the tree 
$\overline{T}_m$, which has the smallest number of leaves and among these trees
  it has the largest Wiener-index. We show that $\overline{T}_m$ must be
  $P_m$. Assume for contradiction that $\overline{T}_m\neq P_m$.

First of all, $\overline{T}_m$ cannot be an image of an LS-switch, because if
$\overline{T_m}$ can be obtained from a tree $F$ by an LS-switch, then by
Theorem~\ref{LS-Th} we have 
$$\hom_0(F(S),P_n)\leq \hom_0(\overline{T}_m(S),P_n),$$ 
and $F$ has at most as many leaves as $\overline{T}_m$,
and $W(F)>W(\overline{T}_m)$. Hence it would contradict the choice of
$\overline{T}_m$.

Let $S$ and $L$ be the small and large class of $\overline{T}_m$,
respectively. Note that $L$ must contain a leaf by Lemma~\ref{leaves}. We will
distinguish two cases according to $S$ containing a leaf too or not.

{\bf Case 1.} $S$ contains a leaf $v$. Let $x$ be the unique
neighbor of $v$. 
First, we show that $x$ has degree at least $3$. 
If $\deg(x)=2$ then let $w$ be the closest vertex to $v$ having degree at
least $3$. Such a $w$ must exist, because $\overline{T}_m$ is not $P_m$.
Note that $d(v,w)\geq 2$. Let us decompose the tree $\overline{T}_m$ into
branches $P(v,w)$, $T_1(w)$ and $T_3(w)$ at the vertex $w$, where  $T_1(w)$
and $T_3(w)$ are non-trivial trees. If $d(v,w)$ is even, then
$\overline{T}_m$ is an image of  a tree $F$ by an LS-switch, where
$T_2(v)=T_4(v)$ are one-vertex trees. If $d(v,w)$ is odd (consequently
$d(x,w)$ is even), then $\overline{T}_m$ is an image of  a tree $F$ by an
LS-switch, where $T_2(x)$ is a one-vertex tree and $T_4(x)$ is the rooted tree
on the vertex set $\{x,v\}$.   In the second case, it may occur that $T_3(w)$
is also an edge implying that $\overline{T}_m=F$. By changing the role of
$T_2(v)$ and $T_4(v)$, we can ensure that $T_1(w)$ is also an edge (otherwise
we get the same contradiction as before). Hence in this case $\overline{T}_m$
is a path with an edge attached on the second vertex. In this case we can
realise that it is an LS-switch of a path since $m$ is odd. Hence we get a
contradiction in this case too.

Now let $u\in L$ be a leaf. Let us delete it, and attach $u'$ to $v$. This way
we get a tree $T'$. By Lemma~\ref{leaf-move} we know that
$$\hom_0(\overline{T}_m(S),P_n)=\hom_0(T'(S),P_n).$$
Furthermore, $T'$ is an LS-switch  (in fact, an
even-KC-transformation) of a tree $F$. Indeed, let us decompose the tree $T'$
at $x$ to  $T_1(x),T_2(x)$ and the path $xvu'$. Thus if we move $T_2(x)$ from
$x$ to $u'$ we get a tree $F$ for which
$$\hom_0(F(S),P_n)\leq  \hom_0(T'(S),P_n).$$
Note that $u'$ and $v$ is not a leaf in $F$ anymore. Maybe, the original
neighbor of $u$ became a leaf in $F$, but still the number of leaves of $F$ is
strictly less than the number of leaves of $\overline{T}_m$. This contradicts
the choice of $\overline{T}_m$. Hence we are done in this case.

{\bf Case 2.} $S$ contains no leaf. 
Hence all leaves belong to
the class $L$. Let $u_1Pu_2$ be a longest path of $\overline{T}_m$. As in the
previous case, the unique neighbors $x_1$ and $x_2$ of $u_1$ and $u_2$,
respectively, have degree at least $3$. We can assume that $x_1\neq x_2$,
otherwise  $\overline{T}_m$ is a star and it contains a claw and we can do
claw-deletion, which strictly decreases $\hom_0(T(S),P_n)$. Since $x_1$ and
$x_2$ have degree at least $3$, and $u_1Pu_2$ were the longest path, the only
possible way it can occur that $x_1$ and $x_2$ have other neighbors $u_3$ and
$u_4$, respectively, which are leaves. Now let us delete $u_3$ and add a new
neighbor $u'_3$ to $x_2$. Let $T'$ be the obtained tree.  Then by
Lemma~\ref{leaf-move}
$$\hom_0(\overline{T}_m(S),P_n)=\hom_0(T'(S),P_n).$$
On the other hand, it is still true that all leaves of $T'$ belong to its
large class. Moreover, it contains a claw: $\{x_2,u_2,u'_3,u_4\}$. 
By Remark~\ref{leaf-rem} the conditions of Lemma~\ref{claw-deletion} are
satisfied, and we can do a claw-deletion. Then we get a tree $F$ for which
$$\hom_0(\overline{T}_m(S),P_n)=\hom_0(T'(S),P_n)>\hom_0(F(S),P_n).$$
This contradicts the choice of $\overline{T}_m$.

Hence we get  contradictions in all cases.
\end{proof}

\section{Open problems}

We collected a few open problems and conjectures in this section.

We first recall a conjecture from the Introduction, namely that there is no
exceptional case in Theorem~\ref{minimality-path} if $n\geq 5$.
\bigskip
 
\noindent \textbf{Conjecture~\ref{disappointment}} 
Let $T_n$ be a tree on $n$ vertices, where
  $n\geq 5$. Then for any tree $T_m$ we have
$$\hom(T_m,P_n)\leq \hom(T_m,T_n).$$
\bigskip

Note that to prove Conjecture~\ref{disappointment}, one only needs to prove
that for any tree $T_m$ we have
$$\hom(T_m,P_n)\leq \hom(T_m,Y_{1,1,n-3})$$
for $n\geq 6$, where $n$ is even.
\bigskip 

There is also an open problem in Figure~\ref{table1}.

\begin{?} Is it true that
$$\hom(P_n,T_n)\leq \hom(T_n,T_n)$$
for every tree $T_n$ on $n$ vertices?
\end{?}

We believe that the answer is affirmative for this question. This question
naturally leads to the following problem.

\begin{?} \label{hard problem} Characterize all graphs $G$ for which
$$\hom(P_m,G)\leq \hom(T_m,G)$$
for all $m$ and all trees $T_m$ on $m$ vertices.
\end{?}

Note that if $G$ is $d$-regular, then $\hom(P_m,G)=\hom(T_m,G)=|V(G)|d^{m-1}$.
We have also seen that the inequality of Problem~\ref{hard problem} is
satisfied if $G=P_n$ or $S_n$. Probably, it is hard to characterize these
graphs. Maybe, it is easier to describe those graphs $G$ for which the
inequality of Problem~\ref{hard problem} is satisfied for large enough $m$.

The dual of Problem~\ref{hard problem} is also natural:

\begin{?} \label{medium problem} Characterize all trees $T_m$ on $m$ vertices
for which 
$$\hom(P_m,G)\leq \hom(T_m,G)$$
for all graph $G$.
\end{?}

Probably, this is an easier problem than Problem~\ref{hard problem}. Note that
already Sidorenko \cite{si} achieved nice results on this problem. Still the
problem is far from being solved.
 
In light of the tree-walk algorithm, it would be interesting to develop an algorithm for computing the number of homomorphisms from bipartite graphs to any graph.

\section*{Acknowledgments}
The first author is very grateful to L\'aszl\'o Lov\'asz and Mikl\'os Simonovits for suggesting to study the papers of Alexander Sidorenko. He also thanks Benjamin Rossman for the comments on the paper 
\cite{br}.

This research was partly done while the second author was visiting the 
Alfr\'ed R\'enyi Institute of Mathematics. He is very grateful to Mikl\'os
Ab\'ert and the Institute for their hospitality and support. He also would like to thank Yanfeng Luo for introducing the graph homomorphisms and Jiang Zeng for all his encouragement during this work.
%He also would like to thank his  supervisors  Yanfeng Luo and Jiang Zeng for their help.   

Special thanks goes to Masao Ishikawa for suggesting the name tree-walk algorithm.

\end{document}